\documentclass[final,3p,times]{elsarticle}

\usepackage{amssymb}
\usepackage{amsthm}

\usepackage{lineno}
\usepackage[colorlinks=true,linkcolor=blue,citecolor=blue,breaklinks]{hyperref}
\usepackage{epsfig} 
\usepackage{epstopdf-base}
\usepackage{svg}
\usepackage{amsmath} 
\usepackage{amssymb}  
\usepackage{amsfonts}
\usepackage{csquotes}
\usepackage{bm}
\usepackage{xcolor}
\usepackage{multirow}
\usepackage{multicol}
\usepackage[squaren]{SIunits}
\usepackage{natbib}
\usepackage[capitalize,nameinlink]{cleveref}
\usepackage{algorithm}
\usepackage{algpseudocode}
\usepackage{pgf}
\usepackage{comment}

\allowdisplaybreaks

\DeclareMathOperator*{\minimize}{min}
\DeclareMathOperator*{\maximize}{max}

\newcommand*\from{\colon}

\theoremstyle{plain}
\newtheorem{thm}{Theorem}
\newtheorem{lem}[thm]{Lemma}

\theoremstyle{definition}
\newtheorem{defn}{Definition}
\newtheorem{ass}{Assumption}

\definecolor{NTNUBlue}{HTML}{00509e}
\definecolor{NTNULightblue}{HTML}{6096d0}
\definecolor{NTNUOrange}{HTML}{ef8114}
\definecolor{NTNUPink}{HTML}{b01b81}
\definecolor{NTNUYellow}{HTML}{f7d019}
\definecolor{NTNUViolet}{HTML}{482776}
\definecolor{NTNUCyan}{HTML}{3cbfbe}
\definecolor{NTNUOcher}{HTML}{cfb887}
\definecolor{LightGrey}{HTML}{bebebe}

\def\changes{black}

\journal{Journal of Process Control}

\modulolinenumbers[5]

\begin{document}

\begin{frontmatter}



\title{Stability Properties of the Adaptive Horizon Multi-Stage MPC}


\author[inst1]{Zawadi Mdoe}
\ead{zawadi.n.mdoe@ntnu.no}

\affiliation[inst1]{organization={Department of Chemical Engineering, Norwegian University of Science and Technology (NTNU)},
            addressline={Sem Sælands vei 4},
            postcode={N-7491},
            city={Trondheim},
            country={Norway}}

\author[inst1]{Dinesh Krishnamoorthy}
\ead{dinesh.krishnamoorthy@ntnu.no}
\cortext[cor1]{Corresponding author}
\author[inst1]{Johannes J{\"a}schke\corref{cor1}}
\ead{johannes.jaschke@ntnu.no}


\begin{abstract}
\textcolor{\changes}{
This paper presents an adaptive horizon multi-stage model-predictive control (MPC) algorithm. 
It establishes appropriate criteria for recursive feasibility and robust stability using the theory of input-to-state practical stability (ISpS).
The proposed algorithm employs parametric nonlinear programming (NLP) sensitivity and terminal ingredients to determine the minimum stabilizing prediction horizon for all the scenarios considered in the subsequent iterations of the multi-stage MPC.
This technique notably decreases the computational cost in nonlinear model-predictive control systems with uncertainty, as they involve solving large and complex optimization problems. 
The efficacy of the controller is illustrated using three numerical examples that illustrate a reduction in computational delay in multi-stage MPC.} 
\end{abstract}


\end{frontmatter}


\section{Introduction}
This paper investigates a way to reduce the computational cost of model predictive control (MPC).
MPC has become increasingly popular in process systems engineering due to its inherent capacity to handle complex constrained multivariate systems \citep{rawlings2017model}.
There has been a broader interest in nonlinear MPC (NMPC) because of the significant nonlinear dynamics in process systems.
However, this requires solving a nonlinear program (NLP) which due to its numerical complexity may result in a large computation cost \cite{murty1985some}.
\textcolor{\changes}{The computation cost is further increased when uncertainty is explicitly considered, in robust MPC approaches, such as in multi-stage MPC formulations.}

\textcolor{\changes}{Multi-stage MPC samples from an uncertain parameter set to obtain a finite set of discrete realizations and then propagate the uncertainty into the future as a growing scenario tree \cite{scokaert1998min,lucia2013robust}.
The size of the optimization problem increases exponentially as the prediction horizon expands, eventually becoming computationally intractable.
As a result, a significant computational delay occurs, affecting control performance.
This poses a major challenge in implementing multi-stage MPC.} 
\textcolor{\changes}{\citet{de2006feedback} and \citet{lucia2013multi} suggest using a robust horizon to decrease the number of variables, but an additional reduction may be required when there is a substantial number of parameter realizations.}
\textcolor{\changes}{Reducing the problem size minimizes computational delay and therefore improves the multi-stage MPC controller performance.}

\textcolor{\changes}{Several methods have been proposed in MPC literature to enhance computational efficiency.
One approach is the adaptive horizon MPC which aims at providing an online horizon update for subsequent MPC iterations.
This framework shortens the prediction horizon as the controlled system approaches a stabilizing region.
The problem size is reduced, leading to shorter solution times while preserving the MPC stability properties.
\citet{krener2018adaptive, Krener2019} proposed using the properties of the Lyapunov function to detect stabilization before shortening the horizon.
This method requires the determination of a good Lyapunov function, which can be difficult for nonlinear systems.
\citet{griffith2018robustly} employs linearization around the operating point to establish a stabilizing region where the horizon can be truncated.
Despite their successful implementation in standard (nominal) MPC, the adaptive horizon techniques have not yet been employed in multi-stage MPC.
Multi-stage MPC has significantly larger solution times than standard MPC because of the explicit consideration of multiple scenarios in its optimization problem.
Hence, the adaptive horizon strategy is even more beneficial for multi-stage MPC, particularly for nonlinear systems.}

\textcolor{\changes}{In this paper, we extend the adaptive horizon scheme proposed in \cite{griffith2018robustly} for multi-stage MPC.
This approach requires performing one-step-ahead predictions, terminal region calculations, and selection of the sufficient (stabilizing) prediction horizon length. 
Although this was studied for the standard MPC in \cite{griffith2018robustly}, the extension of these ideas to multi-stage MPC where we have multiple scenarios is not trivial, because (i) the determination of terminal ingredients depends on the nature of the scenario tree, and (ii) the horizon update using parametric sensitivity needs to ensure that the one-step-ahead approximate solution for each scenario reaches their stabilizing terminal regions.}
\textcolor{\changes}{A proof of concept was very briefly presented in our earlier work \cite{mdoe2021adaptive}. This paper extends our previous work by describing the approach in full detail, providing a robust stability analysis of the resulting controller based on Input-to-State practical Stability (ISpS), and demonstrating the performance on a more complex case study.} 

The remainder of the paper is organized as follows. 
\cref{sec:standardmpc,sec:multistagempc} define the standard MPC and the fixed horizon multi-stage MPC, respectively.
\cref{sec:ahmsnmpc} presents the adaptive horizon multi-stage MPC showing how terminal ingredients are computed, and the horizon update algorithm for multi-stage MPC.
The main results are in \cref{sec:stabilityproperties} where the recursive feasibility and stability properties of adaptive horizon multi-stage MPC are presented for both cases when the scenario tree is fully branched or has a robust horizon.
\cref{sec:casestudies} presents three numerical examples to show the performance of the adaptive horizon multi-stage MPC controller.
The first example demonstrates its control performance with recursive feasibility and robust stability guarantees, and the remaining examples compare the controller performance against the fixed horizon multi-stage MPC for different robust horizons.
Finally, \cref{sec:conclusion} outlines the main findings and future work avenues.

The uncertain process is represented in discrete-time form by the following difference equation:
\begin{equation}
	x_{k+1} = f(x_k, u_k, d_k)
	\label{eq:uncertainNLsystem}
\end{equation}
where $ x_k \in \mathbb{X} \subset \mathbb{R}^{n_x} $ are the states, $ u_k \in \mathbb{U} \subset \mathbb{R}^{n_u} $ are the control inputs, and $ d_k \in \mathbb{D}_c \subset \mathbb{R}^{n_d} $ are the uncertain parameters (disturbances) at time step $ t_k $ where $ k \geq 0 $ is the time index. 
\textcolor{\changes}{The set of real numbers is denoted by $\mathbb R$, and the set of integers by $\mathbb Z$, with the subscript $_+$ referring to their non-negative subspaces.
The Euclidean vector norm is denoted as $ | \cdot | $ and the corresponding induced matrix norm as $ \| \cdot \| $.
The function $ f \from \mathbb{R}^{n_x+n_u+n_d} \mapsto \mathbb{R}^{n_x} $ is assumed to be  twice differentiable in all its arguments, has Lipschitz continuous second derivatives, and is assumed to be transformed to obtain the nominal steady state model as $ f(0,\,0,\,d^0) = 0 $ where $ d^0 $ is the nominal disturbance.}

\section{Standard MPC} \label{sec:standardmpc}
We define the standard MPC as follows.
For system \eqref{eq:uncertainNLsystem} the controller receives the current state $ x_k $ and assumes \textcolor{\changes}{a nominal parameter $d^0$}.
The optimal control problem (OCP) for standard MPC is a parametric NLP in $ x_k $ \textcolor{\changes}{and $d^0$} written as follows:
\begin{subequations}
	\label{eq:standardNMPCformulation}
	\begin{align}
			\textcolor{\changes}{V_N^\mathrm{nom}(x_k, d^0) =} \minimize_{z_i, \nu_i}  \quad & \psi(z_N, \textcolor{\changes}{d^0}) + \sum_{i=0}^{N-1} \ell(z_i, \nu_i, \textcolor{\changes}{d^0}) \label{eq:objstandardNMPCformulation}\\
			\text{s.t.} \quad  & z_{i+1} = f(z_i, \nu_i, \textcolor{\changes}{d^0}), \quad i = 0, \dots, N-1  \label{eq:constraintsstandardNMPCformulation}\\
			& z_0 = x_k, \label{eq:ICstandardNMPCformulation} \\
			& z_i \in \mathbb{X}, \: \nu_i \in \mathbb{U}, \label{eq:boundsstandardNMPCformulation} \\
			& z_N \in \mathbb{X}_f \label{eq:TCstandardNMPCformulation} 
	\end{align}
\end{subequations}
where $ N $ is the prediction horizon, $ V_N^\mathrm{nom} $ is the optimal cost, $ z_i $ and $ \nu_i $ are the state and input variable vectors, respectively at time $ t_{k+i} $.
The objective \eqref{eq:objstandardNMPCformulation} consists of a stage cost $ \ell \from \mathbb{R}^{n_x+n_u+n_d} \mapsto \mathbb{R} $, and a terminal cost $ \psi \from \mathbb{R}^{n_x+n_d} \mapsto \mathbb{R} $.\footnote{\textcolor{\changes}{We have chosen to include the parameter $d$ in the stage cost.
This could, for example, model changes in the setpoint, or other cost function parameters that are set by the user.}}
The initial condition \eqref{eq:ICstandardNMPCformulation} is the current state $ x_k $.
The equations \eqref{eq:constraintsstandardNMPCformulation} are constraints predicting system dynamics, and \eqref{eq:boundsstandardNMPCformulation} are the bounds on state and input variables.
The terminal constraints are in \eqref{eq:TCstandardNMPCformulation}, where $ \mathbb{X}_f \subset \mathbb{X} $ is the terminal region set.

The standard MPC receives $x_k$ and solves \eqref{eq:standardNMPCformulation} for an optimal input sequence $ \bm{v}^\star_{[0,N-1]} $ and their corresponding state predictions $ \mathbf{z}^\star_{(0,N]} $ in the time interval $ (t_k, \, t_{k+N}] $ at every iteration $ k $.
The first stage input from the sequence is applied to the plant as $ u_k $.
Then the prediction horizon is shifted one step, and the problem is re-optimized with the new state $ x_{k+1} $.

Since the current optimal input is obtained with respect to $ x_k $, the standard MPC policy yields a feedback control law $ u_k = \kappa(x_k) $, where $ \kappa\from \mathbb{R}^{n_x} \mapsto \mathbb{R}^{n_u} $.
This gives standard MPC with limited robustness to model uncertainty \citep{de1996robustness, pannocchia2011conditions, yu2014inherent, allan2017inherent}.
However, standard MPC exhibits a substantial decline in performance for significant plant-model mismatch due to uncertainty.

\section{Multi-stage MPC} \label{sec:multistagempc}
\textcolor{\changes}{Unlike standard MPC, multi-stage MPC explicitly takes into account the model uncertainty. 
Multi-stage MPC models uncertainty as a tree of discrete scenarios (called a scenario tree) used to depict the evolution of uncertainty along the prediction horizon \citep{scokaert1998min,lucia2013multi}.
To reduce the problem size in the scenario tree that grows exponentially with the prediction horizon, it is often proposed to limit the branching of the tree only up to a certain stage, called the robust horizon \cite{de2006feedback,lucia2013multi}.
\cref{fig:robusthorizonscenariotree} shows a scenario tree with a robust horizon $ N_R = 2 $ and nine scenarios, where branching stops at $t_{k+N_R} $ and the parameter realizations are kept constant until $t_{k+N}$.} 
\begin{figure}
	\centering
	\includegraphics[scale=1.0]{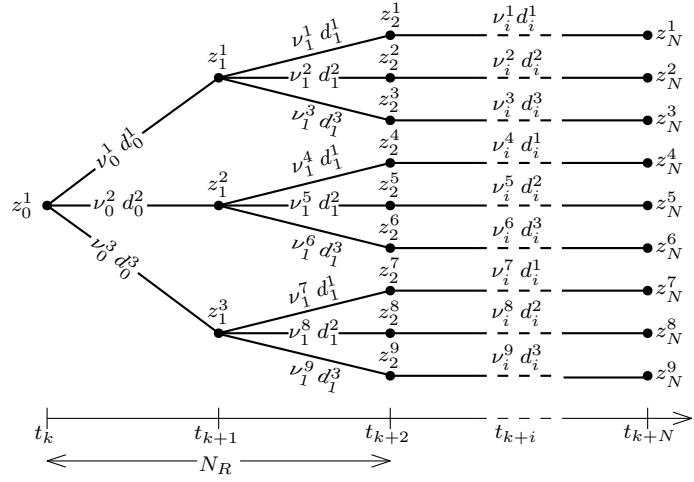}
	\caption{A scenario tree with nine scenarios, prediction horizon $ N $, and robust horizon $N_R = 2$.
	The dashed lines represent all the nodes in between $ t_{k+2} $ and $ t_{k+N} $.}
	\label{fig:robusthorizonscenariotree}
\end{figure}
\textcolor{\changes}{The number of scenarios is given by $\text{card}(\mathbb{C}) = ({N_D})^{N_R}$, where $\mathbb{C}$ is the scenario index set and $N_D$ is the number of discrete parameter realizations.
The problem size is significantly reduced when $N_R << N$, and the fully branched scenario tree is recovered by setting $N_R = N$.}

\textcolor{\changes}{Given $ x_k $ at $ t_k $ multi-stage MPC requires solving the following problem:
\begin{subequations}
	\label{eq:scenarioNLPformulation}
	\begin{align}
			V^\mathrm{ms}_{N}(x_k, \mathbf{d}^c) = \minimize_{z_i^c, \nu_i^c}  \; & \sum_{c\in\mathbb{C}} \omega_c \Big( \psi(z_N^c, d_{N-1}^c) + \sum_{i=0}^{N-1} \ell(z_i^c, \nu_i^c, d_i^c) \Big) \label{eq:scenarioNLPobjective}\\
			\text{s.t.} \;  & z_{i+1}^c = f(z_i^c, \nu_i^c, d_i^c), \; i = 0, \dots, N-1 \\
			& z_0^c = x_k, \\
			& \nu_i^c = \nu_i^{c'}, \; \{(c,c') \; | \; z_i^c = z_i^{c'}\} \label{eq:scenarioNLPnacs} \\
			& d_{i-1}^c = d_i^c, \; i \geq N_R \label{eq:scenarioNLPconstantd}\\
			& z_i^c \in \mathbb{X}, \: \nu_i^c \in \mathbb{U}, \: d_i^c \in \mathbb{D}, \label{eq:scenarioNLPbounds} \\
			& z_N^c \in \mathbb{X}_f^c, \\
			& \forall \, c, c' \in \mathbb{C} \notag
	\end{align}
\end{subequations}}
where $ z_i^c, \nu_i^c $, and $ d_i^c $ represent the state, input, and disturbance vectors at stage $ i $ and scenario $ c $.
The set $ \mathbb C $ is the scenario index set, and the nominal scenario has a zero index.
The objective function \eqref{eq:scenarioNLPobjective} is the weighted sum of the stage and terminal costs across all the scenarios with $ \omega_c $ being the probability for each scenario.
Non-anticipativity constraints (NACs) ensuring that inputs from a common parent node in the scenario tree (\cref{fig:robusthorizonscenariotree}) must be equal are given by \eqref{eq:scenarioNLPnacs}.
This is because it is only possible to make a single control action $ u_k = \nu^\star_0 $ on the plant.
\eqref{eq:scenarioNLPconstantd} is included to show that parameter realizations are constant after the robust horizon.
\textcolor{\changes}{The state and control inputs inequality constraints across all scenarios are given in \eqref{eq:scenarioNLPbounds} and $\mathbb D$ is the set of the discrete parameter realizations.}  

Problem \eqref{eq:scenarioNLPformulation} is a parametric optimization problem in $x_k$ and $\mathbf{d}^c$, where $\mathbf{d}^c$ denotes $\mathbf{d}_{i = 0, \dots, N-1}^{c = 1, \dots, \text{card}(\mathbb C)}$ which is the concatenated vector of all parameter realizations in the scenario tree (across all scenarios and across all time steps in the prediction horizon).

\section{Adaptive horizon multi-stage MPC} \label{sec:ahmsnmpc}
\textcolor{\changes}{Adaptive horizon algorithms provide an update for the prediction horizon length of the MPC problem.
The prediction horizon is not an optimization variable in these approaches, resulting in faster horizon updates compared to solving a mixed-integer NLP (MINLP) in variable horizon MPC approach \cite{shekhar2012variable}.
The adaptive horizon MPC has been implemented for reference tracking in \cite{griffith2018robustly}, and then extended for economic MPC in \cite{krishnamoorthy2020adaptive}. 
In this paper, we aim to extend the algorithm in \cite{griffith2018robustly} to multi-stage MPC for reference tracking objectives.}

\subsection{Variable transformation}
Setpoint tracking under uncertainty implies tracking different implicit references for each scenario.
Therefore, the solution to \eqref{eq:scenarioNLPformulation} has states and control inputs tracking different implicit references for each parameter realization.
It is possible to enforce a common zero by reformulating \eqref{eq:scenarioNLPformulation} with transformed variables.
But first, let us define the following system properties:
\begin{defn}
	(Mapping scenarios to parameter realizations)
    There exists a function $ r \from \mathbb{C} \mapsto \mathbb{Q} $ mapping any scenario $ c $ to its parameter realization index at its leaf node, where $ \mathbb Q $ is the set of parameter realization indices.
\end{defn}
\begin{defn}
	(Stable and optimal equilibrium pairs for each parameter realization)
	For any given parameter realization $ d^r \in \mathbb D $, a state-input pair $ (x^r,\allowbreak u^r) \in \mathbb X \times \mathbb U $ is a \emph{stable equilibrium pair} for system \eqref{eq:uncertainNLsystem} if $ x^r = f(x^r,\allowbreak u^r,\allowbreak d^r) $ holds. 
	Further, if it yields the lowest stage cost among all equilibrium points, then it is the \emph{optimal equilibrium pair} $ (x_f^r,\allowbreak u_f^r) $.
\end{defn}
Then let us make the following assumption.
\begin{ass} \label{ass:steadystatesolution}
	There exists an optimal equilibrium pair $ (x_f^r,\allowbreak u_f^r) $ for each parameter realization $ d^r \in \mathbb D $.
	The optimal equilibrium pairs for each parameter realization are the implicit references for the corresponding scenarios.
\end{ass}
\textcolor{\changes}{\cref{ass:steadystatesolution} implies that there exists a solution to the steady state optimization problem of \eqref{eq:uncertainNLsystem} for each parameter realization $ d^r \in \mathbb D $.}

To enforce a common zero to all the scenarios, reformulate \eqref{eq:scenarioNLPformulation} using the following:
\begin{equation}
	z_i^c = z_i^c - x_f^r, \quad \nu_i^c = \nu_i^c - \nu_f^r
	\label{eq:statetransformation}
\end{equation}
for all $c \in \mathbb C$, and for all $ r \in \mathbb{Q} $.
The transformed system model becomes:
\begin{equation}
	\overline z_{i+1}^c = \overline{f}(\overline z_i^c, \overline \nu_i^c, d_i^c) = f(\overline z_i^c + x_f^r, \overline \nu_i^c + u_f^r, d_i^c) - x_f^r
	\label{eq:odetransformation}
\end{equation}
such that $\overline z_i^c \in \overline{\mathbb X}$, $\overline \nu_i^c \in \overline{\mathbb U}$, where $\overline{\mathbb X}$ and $\overline{\mathbb U}$ are the corresponding new feasible sets of the transformed system.
Equations \eqref{eq:statetransformation} and \eqref{eq:odetransformation} imply when $(\overline z_i^c, \overline \nu_i^c) = (0, 0) $ we have $ x_f^r = f(x_f^r, u_f^r, d_i^c) $ for all $ c \in \mathbb C$, and for all $ r \in \mathbb{Q}$.
We define transformed stage costs as: 
\begin{equation}
	\overline \ell(\overline z_i^c, \overline \nu_i^c, d_i^c) = \ell(\overline z_i^c + x_f^r, \overline \nu_i^c + u_f^r, d_i^c) - \ell(x_f^r, u_f^r, d_i^c)
	\label{eq:stagecosttransformation}
\end{equation}
and transformed terminal costs as:
\begin{equation}
	\overline \psi(\overline z_N^c, d_{N-1}^c) = \psi(\overline z_N^c + x_f^r, d_{N-1}^c) - \psi(x_f^r, d_{N-1}^c)
	\label{eq:terminalcosttransformation}
\end{equation}
such that $\overline \psi(0, d_{N-1}^c) = \overline \ell(0, 0, d_i^c) = 0 $ for all $ c \in \mathbb C $.

The problem \eqref{eq:scenarioNLPformulation} is transformed using \crefrange{eq:statetransformation}{eq:terminalcosttransformation} to obtain a common terminal region containing zero. 
This will be useful to establish closed-loop stability later in \cref{sec:stabilityproperties}.

To keep the notation simple, we will continue using the original notation. 
However, from here on, all variables and functions in equation \eqref{eq:scenarioNLPformulation} are assumed to be transformed.

\textcolor{\changes}{Now that the desired system properties have been defined, the next subsection discusses the methodology to compute the terminal ingredients for each parameter realization in the multi-stage MPC.}
\subsection{Terminal ingredients for the adaptive horizon multi-stage MPC} \label{sec:terminalconditions}
The adaptive horizon multi-stage MPC requires (i) a terminal cost $\psi$, and (ii) a positive invariant terminal region $\mathbb X_f$.
The terminal cost is the weighted average of the terminal costs across all parameter realizations.
For each parameter realization, a terminal cost function is computed beforehand, based on a linearized system about its optimal equilibrium pair \citep{chen1998quasi}.
\color{\changes}
The terminal region is assumed to be the region where a terminal control law on the nonlinear system has a negligible error.
But the determination of the terminal control law depends on whether the scenario tree is fully branched or has a robust horizon. 

\subsubsection{Terminal ingredients for a scenario tree with a robust horizon} \label{sec:terminalingredientsrobusthorizon}
First, consider when the scenario tree has a robust horizon i.e. $N_R < N$. 
The scenario tree has no NACs at the leaf nodes leading to decoupled terminal control input variables ($\nu^c_{N-1}$) across all the scenarios. 
Different stabilizing control laws may exist inside the terminal region for each parameter realization.
Then the terminal regions for each parameter realization are approximated independently. 
We use an infinite horizon LQR for the control law in the terminal region and the terminal cost.
\color{black}
\begin{ass} \label{ass:msnmpctermincond}
	Given a parameter realization $d^r \in \mathbb D$ for all $ r \in \mathbb Q $, there exists a stabilizing LQR with a local control law $ h_f^{r}(x) := u_f^{r}-K_{r} x $ such that $ f(x, h_f^{r} (x), d^r) \in \mathbb X_f^r $ for all $ x \in \mathbb X_f^r$. 
\end{ass}
\textcolor{\changes}{\cref{ass:msnmpctermincond} considers that different LQR controllers may exist in the terminal region across the different scenarios.}

For a given parameter realization $d^r$, linearize \eqref{eq:uncertainNLsystem} about $ (x_f^{r}, u_f^{r}) $ and write the resulting system as a sum of the linear and nonlinear terms as follows:
\begin{equation} \label{eq:nlsystemsplittransformed}
	x_{k+1}  = {A_{K}}_r x_k \, + \phi(x_k, h_f^r(x_k), d^r)
\end{equation}
where $ A_{K_r} = A_r - B_r K_r $, $ A_r = \nabla_x^\top f(0,\allowbreak 0,\allowbreak d^r) $ and $ B_r = \nabla_u^\top f(0,\allowbreak 0,\allowbreak d^r) $ and $ (A_r,\allowbreak B_r) $ is stabilizable, $ \phi \from \mathbb X \times \mathbb U \times \mathbb D \mapsto \mathbb X $ is the nonlinear part of the dynamics.
The infinite horizon LQR for the linearized system with the parameter realization $d^r$ is:
\begin{equation} \label{eq:infinitehorizonLQR}
	\begin{split}
	\psi(x_k, d^r) & = x_k^\top P_r x_k  = \minimize_{\nu_i} \sum_{l=0}^{\infty} (z_i^\top Q z_i \, + \nu_i^\top R \nu_i)\\
	\text{s.t.} \, & z_0 = x_k, \;  z_{i+1} = A_r z_i + B_r\nu_i, \quad \forall \, i = 0,\; \dots, \; \infty
	\end{split}
\end{equation}
where $ Q \succ 0,\, R \succ 0 $ are tuning matrices.
$ P_r \succ 0 $ is the Ricatti matrix for $d^r$.

To determine the terminal region for the LQR on the nonlinear system, there must exist an upper bound on the nonlinear effects for any parameter realization given by \cref{thm:linerrbound}.
\begin{lem} \label{thm:linerrbound}
For any parameter realization $ d^r \in \mathbb{D}$, there exists $ M_r, \, q_r \in \mathbb{R}_+$ such that $ |\phi(x, d^r)| \leq M_r|x|^{q_r}, \, \forall \, x \in \mathbb X $ where $ \phi(x, d^r) := \phi(x, h_f^r(x), d^r) $.
\end{lem}
\begin{proof}
    Consider a single $d^r$ and then follow the proof of Theorem 14 in \cite{griffith2018robustly}.
\end{proof}
Analytical determination of the linearization error and its bound is tedious or impossible.
The bounds are determined offline by one-step simulations instead.
One-step simulations from random initial states are used to obtain the linearization error numerically by evaluating for each $d^r$, the difference $ \phi(x, d^r) = f(x, h_f^r(x), d^r) - A_{K_r}x, \,\forall \, x \in \mathbb X $.
\begin{ass} \label{ass:msnmpcterminalconstraints}
	For a given $ d^r $, there exists a terminal region of attraction $ \mathbb{X}_f^r $  around $ (x^r_f, u_f^r) $ such that $ \mathbb{X}_f^r := \{z \; | \; |z - x^r_f | \leq c_f^r\} $
	where $  c_f^r $ is the terminal radius corresponding to $d^r$.
\end{ass}
\cref{ass:msnmpcterminalconstraints} means that the terminal region set for each parameter realization is approximated by the interior of an \textit{n}-sphere with radius $c_f^r$. 
After fitting $M_r$ and $q_r$ for all parameter realizations according to \cref{thm:linerrbound}, and \cref{ass:msnmpcterminalconstraints} holds, the terminal region radii are computed using \cref{thm:msnmpcterminalradius}. 
\begin{lem} \label{thm:msnmpcterminalradius}
	Suppose \cref{ass:msnmpcterminalconstraints} holds, the terminal radii $ c_f^r $ for all $ d^r \in \mathbb{D} $ depend on their corresponding linearization error bounds which are given by:
	\begin{equation}
		\label{eq:cf}
		c_f^r := \Bigg( \frac{-\bar{\sigma}_r \Lambda_r \, + \sqrt{(\bar{\sigma}_r \Lambda_r)^2 \, + (\underline{\lambda}_{W_r} \, - \epsilon_{LQ})\Lambda_r)}}{\Lambda_r M_r}\Bigg)^\frac{1}{q_r-1}
	\end{equation}
	where $\bar{\sigma}_r$ is the maximum singular value of $ A_{K_r} $, $ \bar{\lambda}_{W_r} $ and $ \underline{\lambda}_{W_r} $ are the maximum and minimum eigenvalues of $W_r := Q + {K_r}^\top RK_r$, $ \Lambda_r:= \frac{\bar{\lambda}_{W_r}}{(1-\bar{\sigma}_r)^2} $, and $ \epsilon_{LQ} > 0$ is a small constant: an allowable tolerance for the terminal cost $ \psi(x, d^r) $.
\end{lem}
\begin{proof}
	As in \cref{thm:linerrbound}, consider a single $d^r$ and use the proof from \cite{griffith2018robustly}.
\end{proof}
\color{\changes}
\subsubsection{Terminal ingredients for a fully branched scenario tree} \label{sec:terminalingredientsfullybranched}
Let us consider the case when the scenario tree is fully branched i.e. $ N_R = N $. 
The stabilizing control law for each parameter realization inside the terminal region cannot be determined independently due to the NACs at the leaf node. 
Therefore, a common terminal control law must exist (i.e. $ h_f^{r}(x) = \kappa_f(x) $, for all $r$).

After linearizing \eqref{eq:uncertainNLsystem} as outlined in \cref{sec:terminalingredientsrobusthorizon}, the following steps are used to determine the terminal ingredients for the adaptive horizon multi-stage MPC with a fully branched scenario tree:
\begin{enumerate}
    \item Select matrices $Q$, and $R$, for the LQR and choose a $K$ such that the common terminal control law $\kappa_f(x) = -Kx$ is stabilizing for all parameter realizations $r$, and $-Kx \in \mathbb U$, $\forall x \in \mathbb X$,
    \item Solve for $P_r$ in the discrete Lyapunov equation $A_r^\top P_rA_r + Q = P_r^\top $, for each parameter realization $r$,
    \item Use $Q$, $R$, $A_{K_r}$ to obtain the terminal radii for each parameter realization $r$ by evaluating \eqref{eq:cf} in \cref{thm:msnmpcterminalradius} where $A_{K_r} = A_r - B_rK$.
    Then pick the smallest as the common terminal region's radius,
    \item Obtain the terminal cost function as the weighted sum $ \psi(x) = \allowbreak \sum_r {x^r}^\top \allowbreak P_r \allowbreak x^r$.
\end{enumerate}
Steps 3 and 4 above are similar procedures as in \cref{sec:terminalingredientsrobusthorizon} only that $K$ is common for all realizations $r$.
\color{black}
\subsection{Problem formulation}
Now that the terminal ingredients for each parameter realization have been defined, we present the formulation for the adaptive horizon multi-stage MPC problem.
The formulation is obtained by replacing the fixed horizon $N$ in \eqref{eq:scenarioNLPformulation} with a variable horizon $N_k$ as follows:
\begin{subequations}
	\label{eq:AHscenarioNLPformulation}
	\begin{align}
		V^\mathrm{ahm}_{N_k}(x_k, \mathbf{d}^c) = & \minimize_{z_i^c, \nu_i^c} \; \sum_{c\in\mathbb{C}} \omega_c\Big(\psi(z_{N_k}^c, d_{N_k-1}^c) + \sum_{i=0}^{N_k-1} \ell(z_i^c, \nu_i^c, d_i^c) \Big) \label{eq:AHscenarioNLPobjective}\\
		\text{s.t.} \; & z_{i+1}^c = f(z_i^c, \nu_i^c, d_i^c), \quad i = 0, \dots, N_k-1 \\
		& z_0^c = x_k, \\
		& \nu_i^c = \nu_i^{c'}, \quad \{(c,c') \; | \; z_i^c = z_i^{c'}\} \label{eq:AHscenarioNLPnacs} \\
		& d_{i-1}^c = d_i^c, \quad i \geq N_R \label{eq:AHscenarioNLPconstantd}\\
		& z_i^c \in \mathbb{X}, \: \nu_i^c \in \mathbb{U}, \: d_i^c \in \mathbb{D}, \\
		& z_N^c \in \mathbb{X}_f^c, \\
		& \forall \, c, c' \in \mathbb{C} \notag
	\end{align}
\end{subequations}
where $	V^\mathrm{ahm}_{N_k} $ is the optimal cost, and $N_k$ is the prediction horizon of the current MPC iteration given by a horizon update algorithm.
The horizon update algorithm requires one-step-ahead predictions obtained from parametric NLP sensitivities that are discussed in the next subsection.

\subsection{NLP sensitivity}
\textcolor{\changes}{Parametric NLP sensitivities are of our interest because problem \eqref{eq:AHscenarioNLPformulation} is parametric in $ x_k $ and $ \mathbf{d}^c $, and we will use them in the horizon update algorithm in \cref{sec:horizonupdatealg}, and for stability analysis in \cref{sec:stabilityproperties}.
Problem \eqref{eq:AHscenarioNLPformulation} can be written as a general parametric NLP in the form:}
\begin{equation} \label{eq:parametricNLP}
    \minimize_w \: J(w, p) \quad \text{s.t.} \quad h(w, p) = 0, \quad g(w, p) \leq 0
\end{equation}
where, $ w \in \mathbb{R}^{n_w} $  is the vector of all optimization variables ($z_i^c$, $\nu_i^c$),
and $ p \in \mathbb{R}^{n_p} $ is the vector of NLP parameters ($ x_k $, $ \mathbf{d}^c $).
The problem \eqref{eq:parametricNLP} has a scalar objective function $ J \from \mathbb{R}^{n_w+n_p} \mapsto \mathbb{R}  $, equality constraints $ h \from \mathbb{R}^{n_w+n_p} \mapsto \mathbb{R}^{n_e} $, and inequality constraints $ g \from \mathbb{R}^{n_w+n_p} \mapsto \mathbb{R}^{n_i} $.
Then the function $\mathcal{L}(w, \lambda, \mu, p) = J(w, p) + \lambda^\top h(w, p) + \mu^\top g(w, p)$ 
is the Lagrange function of \eqref{eq:parametricNLP} where, $ \lambda $ and $ \mu $ are the Lagrange multipliers of appropriate dimension.

A point $ w^\star $ that \eqref{eq:parametricNLP} satisfies the \textit{Karush-Kuhn-Tucker} conditions is known as a KKT point \citep{nocedal2006numerical}.
Given a parameter $ p $, and a point $ w^\star $ that satisfy:
\begin{align} \label{eq:KKTconditions}
	\begin{split} 
		\nabla_{w} \mathcal{L}(w^\star, \lambda^\star, \mu^\star, p) = 0 \\
		h(w^\star, p) = 0 \\
		g(w^\star, p) \leq 0 \\
		0 \leq \mu^\star \: \bot \: g(w^\star, p) \leq 0
	\end{split}	 
\end{align}
for some multipliers $ (\lambda, \mu) $, where $ w^\star $ is called a KKT point. 
The set of all multipliers $ \lambda $ and $ \mu $ that satisfy the KKT conditions for a certain parameter $ p $ is denoted as $ \mathcal{M}(p) $.
The active constraint set is given by $ \mathcal{A}(w^\star) = \{j \; | \; g_j(w^\star, p) = 0\} $.
$\perp$ is the complementarity operator, i.e. either $\mu$ or $g$ (or both) must be zero.
Then we define strict complementarity as follows.
\begin{defn}
	(SC, \citep{nocedal2006numerical}) At the KKT point $ w^\star $ of \eqref{eq:parametricNLP} with multipliers ($ \lambda^\star $, $ \mu^\star $), strict complementarity (SC) condition holds if $ \mu_j + g_j(w^\star, p) > 0 $ for all $ j \in \mathcal A(w^\star) $.
\end{defn}
For the KKT conditions to be first-order necessary optimality conditions, the KKT point $ w^\star $ requires a constraint qualification to be satisfied.
Further, it requires a constraint qualification to be a local minimizer of \eqref{eq:parametricNLP}.
The following is the definition of one well-known constraint qualification.
\begin{defn}
	(LICQ, \citep{nocedal2006numerical}) The linear independence constraint qualification (LICQ) holds at $ w^\star $ when the gradient vectors $\nabla h_i(w^\star, p)$ {and} $\nabla g_j(w^\star, p)$ for all $ i$
	, and for all $ j \in \mathcal A(w^\star)$ are linearly independent.
	LICQ implies that the multipliers $ \lambda^\star $, $ \mu^\star $ are unique i.e. $ \mathcal{M}(p) $ is a singleton.
\end{defn}
Given LICQ, the KKT conditions are necessary but not sufficient conditions for optimality.
A second-order condition is needed to guarantee a minimizer:
\begin{defn}
	(SSOSC, \citep{fiacco1976sensitivity}) The strong second-order sufficient condition (SSOSC) holds at $ w^\star $ with multipliers  $ \lambda^\star $ and $ \mu^\star $ if $q^\top \nabla_{ww}^2 \mathcal{L}(w^\star, \lambda^\star, \mu^\star, p)q > 0$ for all vectors $q \neq 0$,
	such that $	\nabla h_i(w^\star, p)^\top q = 0$, for all $i$ and $\nabla g_j(w^\star, p)^\top q = 0$ for all $ j \in \mathcal{A}(w^\star) \cap \{j \, | \, \mu_j > 0\}$.
\end{defn}
Assuming SC, LICQ, and SSOSC and applying the Implicit Function Theorem (IFT) to the KKT conditions in \eqref{eq:KKTconditions} leads to \cref{thm:IFTtoKKT}.
\begin{thm} \label{thm:IFTtoKKT}
	(IFT applied to KKT conditions)
	Assume a KKT point $w^\star(p_0) $ that satisfies \eqref{eq:KKTconditions}, and that SC, LICQ and SSOSC hold at $ w^\star(p_0) $.
	Furthermore, the functions $ F $, $ h $, and $ g $ are at least $ k + 1 $ times differentiable in $ w $ and $ k $ times differentiable in $ p $.
	Then the primal-dual solution $ s^\star(p)^\top = [ w^\star(p)^\top, \lambda^\star(p)^\top, \mu^\star(p)^\top ] $ has the following are the properties:
	\begin{itemize}
		\item $ s^\star(p_0) $ is an isolated local minimizer of \eqref{eq:parametricNLP} at $ p_0 $ and contains unique multipliers $ (\lambda^\star(p_0),\: \mu^\star(p_0)) $.
		\item For $ p $ in a neighborhood of $ p_0 $ the active constraint set $ \mathcal{A}(w^\star) $ remains unchanged.
		\item For $ p $ in a neighborhood of $ p_0 $ there exists a unique, continuous and differentiable function $ s^\star(p) $ which is a local minimizer satisfying SSOSC and LICQ for \eqref{eq:parametricNLP} at $ p $.
		\item There exists positive Lipschitz constants $ L_s $, $ L_v $ such that $ |s^\star(p) - s^\star(p_0) | \leq L_s | p - p_0 | $ and $ | J(p) - J(p_0) | \leq L_v | p - p_0 | $, where $|\cdot| $ is the Euclidean norm.
	\end{itemize}
\end{thm}
\begin{proof}
	see \citet{fiacco1976sensitivity}.
\end{proof}
\textcolor{\changes}{The theorem above establishes Lipschitz continuity in the optimal solution and optimal objective function of \eqref{eq:AHscenarioNLPformulation} with respect to $x_k$ and $\mathbf{d}^c$.
This result is important for the stability analysis.}
\cref{thm:IFTtoKKT} ensures we can always obtain from \eqref{eq:KKTconditions} the following linear system for NLP sensitivity if SC, LICQ, and SSOSC hold at a KKT point $ w^\star(p) $,
\begin{equation} \label{eq:KKTmatrixsens}
	\begin{bmatrix}
		\nabla_{ww}^2 \mathcal L 		& \nabla h	& \nabla g_{\mathcal A} \\
		\nabla h^\top  	                & 0 	    & 0 					\\
		\nabla g_{\mathcal A}^\top  	& 0 		& 0 						
	\end{bmatrix}
	\begin{bmatrix}
		\nabla_{p} w^\star 		\\
		\nabla_{p} \lambda^\star 	\\
		\nabla_{p} \mu^\star_\mathcal{A}						
	\end{bmatrix}
	= \begin{bmatrix}
		\nabla_{wp}^2 \mathcal L \\
		\nabla h^\top  \\
		\nabla g_{\mathcal A}^\top 
	\end{bmatrix}
\end{equation}
where all the derivatives of are evaluated at $ w^\star $, $ p_0 $, $ (\lambda^\star, \mu^\star) \in \mathcal{M}(p_0) $, and $ g_{\mathcal A} $ and $ \mu^\star_{\mathcal A} $ are vectors with elements $ g_j $ and $ \mu^\star_j $ for all $ j \in \mathcal A(w^\star) $, respectively.
The Taylor expansion of the primal solution at $p_0$ is $ w^\star(p) = w^\star(p_0) + {\nabla_{p}w^\star}^\top(p - p_0) + \mathcal O (|p - p_0|^2) $.
As a result, a first-order approximation is obtained in the neighborhood of $p_0$ using:
\begin{equation} \label{eq:nlpsenslinearapprox}
    w^\star(p) \approx w^\star(p_0) + {\nabla_{p}w^\star}^\top(p - p_0)
\end{equation}
\textcolor{\changes}{The KKT matrix in \eqref{eq:KKTmatrixsens} is evaluated at $p_0$ and is obtained for free from the previous solution.
The benefit of the sensitivity step \eqref{eq:nlpsenslinearapprox} is to perform one-step-ahead predictions of the solution at the subsequent MPC iteration with a minimal computational effort compared to solving the NLP from scratch.
The following subsection presents the horizon update algorithm for adaptive horizon multi-stage MPC using the sensitivity update.}
\subsection{Horizon update algorithm} \label{sec:horizonupdatealg}
The algorithm begins by selecting a long horizon $ N_\text{max} $, a safety factor $N_{\text{min}} > N_R$, and computing the terminal ingredients offline.
In each MPC iteration, the algorithm determines a horizon update such that all scenarios reach their terminal regions irrespective of the initial state in the subsequent MPC iteration. 
\begin{figure}
	\centering
	\includegraphics[width=0.7\linewidth]{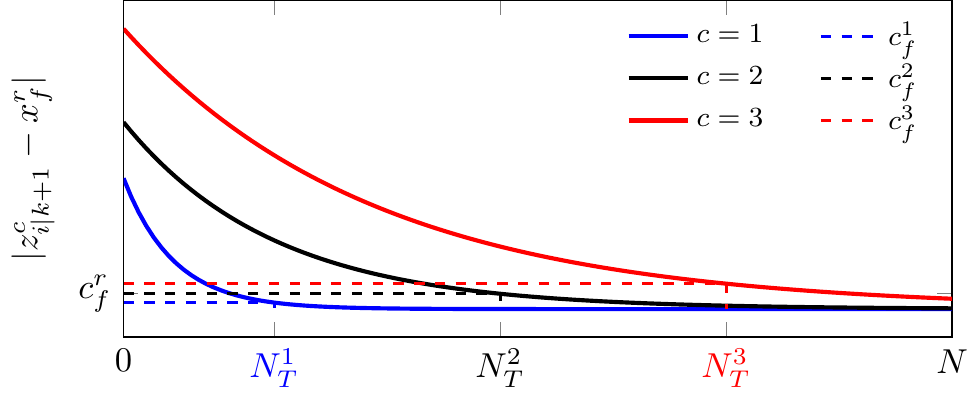}
	\caption{Sensitivity prediction (solid lines) of \eqref{eq:AHscenarioNLPformulation} for three scenarios at $t_{k+1}$, and the terminal regions (below dashed lines). 
	$N_{k+1}$ will be reduced from $ N $ to at least $ N_T^3 $.}
	\label{fig:trajectory}
\end{figure}
\cref{fig:trajectory} is a simple illustration of how a new horizon length is determined for three scenarios.
The one-step-ahead prediction of the solution to \eqref{eq:AHscenarioNLPformulation} from a sensitivity update is in solid lines and the terminal regions for each scenario are below the dashed lines.
Each scenario reaches its respective terminal region at stage $N_T^c$. 
To ensure that each scenario reaches its terminal region at $ t_{k+1} $, the new prediction horizon must be at least equal to the largest $N_T^c$.
Therefore in \cref{fig:trajectory} $N_{k+1}$ must be greater than $N_T^3$. 

\cref{alg:NLPsensalg} presents the horizon update algorithm for multi-stage MPC.
\begin{algorithm}
\caption{Horizon update for multi-stage MPC}\label{alg:NLPsensalg}
\begin{algorithmic}[1]
\State Define $N_\text{max}$, $N_\text{min} > N_R$
\State Determine $\mathbb X_f^r$, $P_r$ for all $d^r \in \mathbb D$
\State Initialize: $k \gets 0$, $N_0 \gets N_\text{max}$
\State $w^\star(x_k) \gets$ Solve \eqref{eq:AHscenarioNLPformulation} with $N_k$ and initial state $x_k$
\State $ \mathcal X_{k+1} \gets \{x_{k+1}^r \: | \: x_{k+1}^r = f(x_k, u_k, d^r)\: \forall d^r \in \mathbb D\}$
\State $\nabla_{x_k}w^\star \gets $ Evaluate \eqref{eq:KKTmatrixsens} to obtain NLP sensitivities at $x_{k}$ 
\State $\mathcal W_{k+1}^\star \gets \{\text{get} \: w^\star(x_{k+1}^r) \: \text{using \eqref{eq:nlpsenslinearapprox}} \: \forall x_{k+1}^r \in \mathcal X_{k+1}\}$ 
\ForAll {$w^\star(x_{k+1}^r) \in  W_{k+1}^\star$} 
    \If{$z^c_{N_k|k+1}\in \mathbb{X}_f^r \: \forall c \in \mathbb C$} 
        \State $N_T^c \gets $ step at which $\mathbb{X}_f^r$ is reached
        \State $N_T(x_{k+1}^r) \gets \maximize_{c\in\mathbb C} N_T^c$ 
    \Else 
        \State $N_T(x_{k+1}^r) \gets N_\text{max}$
    \EndIf
\EndFor
\State $N_T \gets \maximize_{x_{k+1}^r \in \mathcal X_{k+1}} N_T(x_{k+1}^r) $
\State $N_{k+1} \gets \minimize(N_\text{max},  N_T + N_\text{min})$
\State $k \gets k+1$
\State Go to Step 4 
\end{algorithmic}
\end{algorithm}
The flowchart in \cref{fig:NLPsensalg} summarizes \cref{alg:NLPsensalg} which is a mapping of the current state and prediction horizon to the subsequent prediction horizon.
The mapping is defined as a function in \cref{def:horizonupdate} as follows:
\begin{figure}
	\centering
	\includegraphics[scale=1.0]{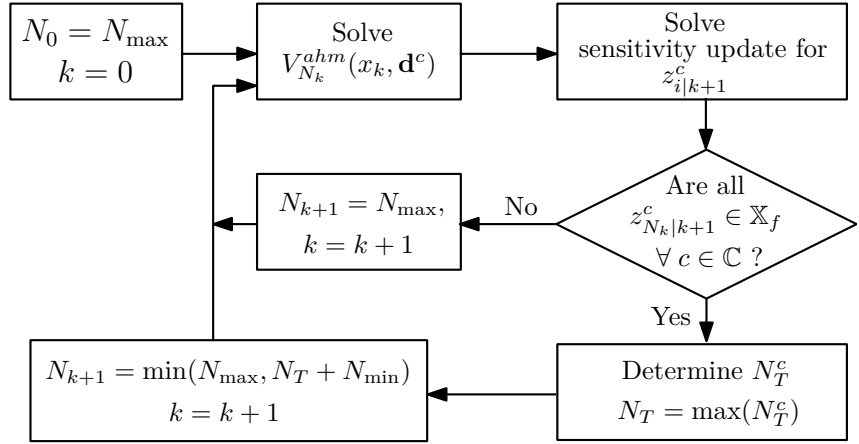}
	\caption{Horizon update algorithm for multi-stage MPC that gives admissible prediction horizons for subsequent MPC iterations.}
	\label{fig:NLPsensalg}
\end{figure}
\begin{defn} \label{def:horizonupdate}
	(Horizon update function)
	Let $ \mathcal N $ be the universal set of admissible horizon lengths such that $ \mathcal N = \{ N \,|\,N_\text{min} \leq N \leq N_\text{max}, \, N \in \mathbb{Z}_+\} $.
	Further, there exists a subset of horizon lengths $ \mathcal N_k \subset \mathcal N $ that define the feasible problems \eqref{eq:AHscenarioNLPformulation} at time $t_k$.
	We define a function $ H \from \mathbb{R}^{n_x} \times \mathcal N \times \mathbb{R}^{n_x} \mapsto \mathcal N $ that determines the horizon lengths for the adaptive horizon multi-stage algorithm such that $ N_{k+1} = H(x_k, N_k, z^c_{k+1|k}) \in \mathcal N_{k+1} $.
\end{defn}
\section{Stability properties of the adaptive horizon multi-stage MPC} \label{sec:stabilityproperties}
\textcolor{\changes}{Applying the optimal control input computed by the MPC controller to system \eqref{eq:uncertainNLsystem} results in a closed-loop uncertain system whose dynamics are denoted by $x_{k+1} = F(x_k, d_k)$. 
We will perform a stability analysis for this closed-loop uncertain system.
Input-to-state practical stability (ISpS) is the suitable framework to analyze the stability properties of such a system. We start with recollecting some definitions in ISpS stability analysis of discrete-time nonlinear systems:}
\begin{defn}
	(Comparison functions \citep{SONTAG1989}) \textcolor{\changes}{A function $ \alpha \from \mathbb{R}_+ \mapsto \mathbb{R}_+ $ is said to be of class $ \mathcal{K} $ if $ \alpha(0) = 0 $, $ \alpha(n) > 0 $ for all $ n > 0 $ and is strictly increasing.
	It becomes of class $ \mathcal{K}_\infty $ if in addition, $ \alpha $ is unbounded.
	A function $ \beta \from \mathbb{R}_+ \times \mathbb{Z}_+ \mapsto \mathbb{R}_+ $ is of class $ \mathcal{KL} $ if $ \beta(\cdot,k) $ is of class $ \mathcal{K} $ for each fixed $ k\geq0 $, and $ \beta(n, \cdot) $ is decreasing for each fixed $ n\geq0 $ and $ \beta(n,k) \rightarrow 0 $ as $ k \rightarrow \infty $.}
\end{defn}
\begin{defn}
	\textcolor{\changes}{(Regional input-to-state practical stability (ISpS), \cite{sontag1996new})}
	An autonomous system $ x^+ = F(x,d) $ is ISpS in $ \Gamma \subseteq \mathbb R^n $ and $ 0 \in \Gamma $ if there exists a $ \mathcal{KL} $ function $ \beta $, a $ \mathcal K $ function $ \gamma $, \textcolor{\changes}{and a constant $ c \geq 0 $ such that for all $ k\geq0 $
	\begin{equation} \label{eq:ISpS}
		|x_k| \leq \beta(|x_0|,k) + \gamma(\|\mathbf{d}_{[0, k-1]} - \mathbf{d}^0\|) + c, \quad \forall \, x_0 \in \Gamma
	\end{equation}}
	where $\mathbf{d}_{[0, k-1]}:= \begin{bmatrix} d_0, & d_1, & \dots, & d_{k-1} \end{bmatrix}$ is the sequence of true parameter realizations and $ \mathbf{d}^0 $ is a sequence of nominal parameters of the same length.
\end{defn}
\begin{defn}
	\textcolor{\changes}{(ISpS Lyapunov function in $\Gamma$, \cite{limon2006input})}
	Suppose $ \Gamma $ is an RPI set and there exists a compact set $\Theta \subseteq \Gamma$ with the origin as an interior point.
	A function $ V \from \mathbb R^n \mapsto \mathbb{R}_+ $ \textcolor{\changes}{is an ISpS Lyapunov function in $\Gamma$} for the system $ x^+ = F(x,d) $, if there exist $ \mathcal{K}_\infty $ functions $ \alpha_1 $, $ \alpha_2 $, $ \alpha_3 $, $ \mathcal{K} $ function $ \sigma $, and constants $ c_1, c_2 \geq 0 $ such that:
	\begin{subequations} \label{eq:ISpSLyapunovFunction}
		\begin{align}
			V(x) & \geq \alpha_1(|x|), \quad \forall \, x \in \Gamma \label{eq:ISpSlowerbound}\\
			V(x) & \leq \alpha_2(|x|) + c_1,  \quad \forall \, x \in \Theta \label{eq:ISpSupperbound}\\
			\begin{split}
			V(x^+) - V(x) & \leq -\alpha_3(|x|) + \sigma(|d - d^0|) + c_2,  \\
			 & \quad \forall \, x \in \Gamma, \forall \, d, d^0 \in \Omega. 
			\end{split}\label{eq:ISpSdescent}
		\end{align}
	\end{subequations}
 \textcolor{\changes}{where $\Omega$ is the uncertain parameter (disturbance) set.}
\end{defn}
A standard MPC equivalent with an adaptive horizon can be obtained from \eqref{eq:AHscenarioNLPformulation} by setting $ d^c_i = d^0_i $ for all $ c \in \mathbb C $ as follows:
\begin{subequations}
	\label{eq:AHscenarioNLPnominalequiv}
	\begin{align}
		V^\mathrm{ahm}_{N_k}(x_k, \mathbf{d}^0) = & \minimize_{z_i^c, \nu_i^c}  \: \sum_{c\in\mathbb{C}} \omega_c \Big( \psi(z_{N_k}^c, d_{N_k-1}^0) + \sum_{i=0}^{N_k-1} \ell(z_i^c, \nu_i^c, d_i^0) \Big)  \label{eq:AHscenarioNLPobjectivenominalequiv} \\
		\text{s.t.} \quad & z_{i+1}^c = f(z_i^c, \nu_i^c, d_i^0), \quad i = 0, \dots, N_k-1 \\
		& z_0^c = x_k, \\
		& \nu_i^c = \nu_i^{c'}, \quad \{(c,c') \; | \; z_i^c = z_i^{c'}\} \label{eq:AHscenarioNLPnacsnominalequiv} \\
		& d_{i-1}^0 = d_i^0, \quad i = N_R, \; \dots,\; N_k-1 \label{eq:AHscenarioNLPconstantdnominalequiv}\\
		& z_i^c \in \mathbb X, \: \nu_i^c \in \mathbb U, \: d_i^0 \in \mathbb{D}, \\
        & z_{N_k}^c \in \mathbb X_f^c, \\
		& \forall \, c, c' \in \mathbb{C}, \ \notag
	\end{align}
\end{subequations}
where \eqref{eq:AHscenarioNLPnominalequiv} effectively consists of $ c $ identical copies of the standard MPC with an adaptive horizon.

It follows from \cref{thm:IFTtoKKT} that the difference between the primal solutions $ (z_i,\allowbreak \nu_i) $ of \eqref{eq:AHscenarioNLPformulation} and \eqref{eq:AHscenarioNLPnominalequiv} is $|w^\star(\mathbf{d}^c) -w^\star(\mathbf{d}^0)| \leq L_s|\Delta \mathbf{d}| $, where $ w^\star = [{z^\star}^\top, {\nu^\star}^\top]^\top$, $ L_s > 0 $ is the Lipschitz constant, and $ |\Delta \mathbf{d}| $ is defined as:
\begin{equation} \label{eq:deltaddefinition}
     |\Delta \mathbf{d}| := \maximize_{d^r,d^{r'}\in \mathbb D} |d^r - d^{r'}|
\end{equation}
that is the maximum difference in the disturbance vector between any two parameter realizations.
Since the evolution of \eqref{eq:uncertainNLsystem} in closed-loop depends on the true disturbance $ d_k $, then $ x_{k+1} = f(x_k,\allowbreak\kappa(x_k),\allowbreak d_k)$, and $ x_{k+1} \leq f(x_k,\allowbreak \kappa(x_k),\allowbreak d_k^0) + \mathcal O(|\Delta \mathbf{d}|) $.
This result facilitates the recursive feasibility and ISpS stability analysis for the adaptive horizon multi-stage MPC that follows.
But first, let us define the following basic assumptions:
\begin{ass} \label{ass:basicstabilityassumptions}
	(Basic assumptions for adaptive horizon multi-stage MPC)
	The adaptive horizon multi-stage MPC has the following properties:
	\begin{enumerate}[A.]
		\item Lipschitz continuity: The functions $ f$, $ \ell $, and $ \psi $ are Lipschitz continuous with respect to $x$, $u$ and $d$ in the compact set $ \mathbb X \times \mathbb U \times \text{convhull}(\mathbb D) $. \label{itm:basicstabilityassumptions1}
		\item Constraint set: The constraint sets $ \mathbb X $ and $ \mathbb U $ are \textcolor{\changes}{compact, and
		contain the origin in their interiors.}
		\item \textcolor{\changes}{The solution to \eqref{eq:AHscenarioNLPformulation} satisfies LICQ and SSOSC such that \cref{thm:IFTtoKKT} applies.}
	\end{enumerate}
\end{ass}
 \textcolor{\changes}{With these properties in \cref{ass:basicstabilityassumptions}, the Lipschitz continuity holds w.r.t. $x_k$ and $\mathbf{d}^c$ for the optimal solution and the optimal cost $V_{N_k}^{ahm}$.\footnote{\textcolor{\changes}{The LICQ condition may be relaxed to the Mangasarian-Fromovitz constraint qualification (MFCQ), the constant rank constraint qualification (CRCQ), and the general strong second-order sufficient condition (GSSOSC). In that case, a similar Lipschitz argument can be made and applied in the following analysis. \cite{ralph1995directional,jaschke2014fast}}}}
\begin{ass} \label{ass:stagecostproperty}
	(Property of the stage cost)
	Given a common control law $\kappa \from \mathbb X \mapsto \mathbb U$ then for any parameter realization $d^r \in \mathbb{D} $, the stage cost $\ell $ is bounded as follows: 
	$ \alpha_p(|x|) \leq \ell(x,\allowbreak \kappa(x),\allowbreak d^r) \leq \alpha_q(|x|) + \sigma_q(|d^r-d^0|)$ where $ \alpha_p(\cdot),\allowbreak \; \alpha_q(\cdot) \in \mathcal K_\infty $, and $\sigma_q(\cdot) \in \mathcal K$.
\end{ass}
\textcolor{\changes}{\cref{ass:stagecostproperty} requires that the stage cost is positive, with lower and upper bounds proportional to the state at the nominal disturbance.
When the disturbance realization is not nominal, the upper bound increases by a term proportional to the distance of the disturbance realization from the nominal.}
\begin{ass} \label{ass:terminconditionsahmsnmpc}
	(Assumptions on terminal conditions of the adaptive horizon multi-stage MPC).
	Let \cref{ass:basicstabilityassumptions} hold for all parameter realizations $d^r \in \mathbb{D} $ then:
	\begin{enumerate}[A.]
		\item There exists a common terminal region $ \mathbb{X}_f = \bigcap\limits_{r \in \mathbb{Q}}\mathbb{X}_f^r $, that is \textcolor{\changes}{robust control invariant.}   \label{itm:terminconditionsahmsnmpcA}
		\item $ \mathbb{X}_f \subseteq \mathbb{X} $ is \textcolor{\changes}{compact and contains the origin in its interior.} 
		\item The terminal cost is bounded as follows: $ \alpha_{p,\allowbreak\psi}(|x|)  \leq \psi(x,\allowbreak d^r) \leq \alpha_{q,\allowbreak\psi}(|x|) + \sigma_{q,\allowbreak\psi}(|d^r - d^0|)$ where $\alpha_{p,\allowbreak\psi},\allowbreak \alpha_{q,\allowbreak\psi} \in \mathcal K_\infty$, and $\sigma_{q,\psi} \in \mathcal K$. \label{itm:terminconditionsahmsnmpc3} 
		\item The terminal cost is a local control Lyapunov function $ \psi(f(x,\allowbreak \kappa_f(x),\allowbreak d^r),\allowbreak d^r) - \psi(x,\allowbreak d^r) \leq - \ell(x,\allowbreak \kappa_f(x),\allowbreak d^r)$ for all $ x \in \mathbb{X}_f $.
	\end{enumerate}
\end{ass}
\color{\changes}
\cref{ass:terminconditionsahmsnmpc}\ref{itm:terminconditionsahmsnmpcA} is needed because a common terminal region must exist for a possible reduction of the prediction horizon.
The selection of a prediction horizon where all scenarios reach their respective terminal regions in Step 11 of  \cref{alg:NLPsensalg} implies that the common terminal region is reached.
This region may be a null set when $|\Delta \mathbf d|$ is sufficiently large.
Further, this can be determined beforehand when computing the terminal ingredients.

\subsection{Recursive feasibility of the adaptive horizon multi-stage MPC with a fully branched scenario tree} \label{sec:recursivefeasibilityfullybranched}
Before stability analysis, recursive feasibility must be guaranteed by ensuring robust constraint satisfaction.
Given the assumptions above, we first present the following result for a fully branched scenario tree i.e. $N_R = N_k$:
\begin{thm} \label{thm:recursivefeasibility}
	(Recursive feasibility of the adaptive horizon multi-stage MPC).
	Suppose \cref{ass:basicstabilityassumptions}, \ref{ass:stagecostproperty} and \ref{ass:terminconditionsahmsnmpc} hold, then problem \eqref{eq:AHscenarioNLPformulation} from a fully branched scenario tree $N_R = N_k $ is recursively feasible.
\end{thm}
\begin{proof} 
Let \cref{ass:basicstabilityassumptions} and \ref{ass:terminconditionsahmsnmpc} hold and consider the fully branched multi-stage MPC with $N_R = N_k$.
Proposition 4 in \cite{maiworm2015scenario} and Theorem 6 in \cite{lucia2020stability} show that if a robust control invariant terminal region exists then the fully branched multi-stage MPC with a fixed horizon $N_R = N_k$ is recursively feasible.
Since the adaptive horizon, multi-stage MPC with $N_R = N_k$ ensures that the horizon is updated such that the common terminal region is always reached, it is also recursively feasible for any horizon update $N_{k+1}$ given by \cref{alg:NLPsensalg}.
\end{proof}

\subsection{Recursive feasibility of the adaptive horizon multi-stage MPC with a robust horizon} \label{sec:recursivefeasibilityrobusthorizon}
Recursive feasibility cannot be guaranteed by using the same arguments in \cref{thm:recursivefeasibility} when the scenario tree is not fully branched.
But the infeasibility of problem \eqref{eq:AHscenarioNLPformulation} when $N_R < N_k$ may be avoided by relaxing the state inequalities using soft constraints \cite{thombre2021sensitivity}.
However, robust constraint satisfaction is not guaranteed because the state trajectories may cross the feasible set $\mathbb X$.
Therefore, before proceeding to the stability analysis the following assumption must be made:
\begin{ass} \label{ass:recursivefeasibility}
    Problem \eqref{eq:AHscenarioNLPformulation} is feasible at time $t_k$ when $N_k=N_\text{max}$.
    Furthermore, if problem \eqref{eq:AHscenarioNLPformulation} at time $t_k$ with $x_k$ and $N_k$ is feasible, then so is problem \eqref{eq:AHscenarioNLPformulation} solved at time $t_{k+1}$ with $x_{k+1}$ and $ N_{k+1} = H(x_k, N_k, x_{k+1}) \in \mathcal N_{k+1} $ for all $x_k,x_{k+1} \in \mathbb X$, $N_k \in \mathcal{N}_{k+1}$.
\end{ass}
Note that \cref{ass:recursivefeasibility} can always be satisfied by choosing $N_R = N_k$. (cf. Theorem 4). The increase in the problem size with a long robust horizon required to satisfy \cref{ass:recursivefeasibility} also motivates the need for shortening the horizon length using our proposed \cref{alg:NLPsensalg}.
\cref{ass:recursivefeasibility} ensures that $H$ (see \cref{def:horizonupdate}) produces feasible horizon lengths for any robust horizon $N_R \leq N_k$, essentially implying recursive feasibility.
This allows us to present results on the ISpS property of the adaptive horizon multi-stage MPC.
\color{black}
\subsection{ISpS for adaptive horizon multi-stage MPC}
First, the relationships between the linear and nonlinear systems are needed because the nonlinear system \eqref{eq:uncertainNLsystem} is controlled by a stabilizing linear controller inside the common terminal region $\mathbb X_f$.
Consider the nominal problem \eqref{eq:AHscenarioNLPnominalequiv} inside the terminal region evaluated using a stabilizing terminal control law $ \kappa_f^0 $:
\begin{equation}\label{eq:LQRcostnominalequiv}
	V_N^\mathrm{lqr}(x_k, \mathbf{d}^0) :=  \sum_{c \in \mathbb C} \omega_c \Big (\psi(z^0_N, d_{N-1}^0) + \sum^{N-1}_{i=0} \ell(z^0_i, \kappa_f^0(z^0_i), d^0_i) \Big) \\
\end{equation}
where $z^0_{i+1} = A_0z^0_i + B_0\nu^0_i + \psi(z^0_i, d^0), \; \forall i = 0, \dots, N-1$, and $ V_N^{lqr}(x_k, \mathbf{d}^0) $ is the nominal LQR cost.
\textcolor{\changes}{The relationships between the optimal costs of the nominal equivalents of the LQR and adaptive horizon multi-stage MPC in the terminal region are given in \cref{lem:effectoflinoncosts}.}
\begin{lem} \label{lem:effectoflinoncosts}
	(Effect of linearization on costs inside the terminal region)
	There exists $ \alpha_n \in \mathcal K_\infty $ such that $ |V^\mathrm{lqr}_N(x, \mathbf{d}^0) - \psi(x, d^0)| \leq \alpha_n(|x|) $ and $ \alpha_v \in \mathcal K_\infty $  such that  $ | \psi(x, d^0) - V_N^\mathrm{ahm}(x, \mathbf{d}^0)| \leq \alpha_v(|x|)$ for all $ x \in \mathbb X_f $ and for all $N \in \mathcal N$. 
\end{lem}
\begin{proof}
	Because this is a nominal problem, the result holds because of Lemmas 18 and 19 in \cite{griffith2018robustly}.
\end{proof}
\begin{ass} \label{ass:nominalahnmpcnleffects}
	The solution to \eqref{eq:AHscenarioNLPnominalequiv} with a prediction horizon $ N_k \geq N_\text{min} $ satisfies:
	\begin{subequations}
		\begin{align}
			\alpha_n(|z^0_{N_k|k}|) - \alpha_p(|x_k|) & \leq -\alpha_3(|x_k|) \text{ if } N_{k+1} \geq N_k \label{eq:nominalahnmpceffects1}\\
			\alpha_v(|z^0_{N_{k+1}+1|k}|) - \alpha_p(|x_k|)  & \leq -\alpha_3(|x_k|) \text{ if } N_{k+1} < N_k \label{eq:nominalahnmpceffects2}
		\end{align}
	\end{subequations}
	for some $ \alpha_3 \in \mathcal K_\infty $, where $ N_{k+1} = H(x_k, N_k, z^0_{k+1|k}) \in \mathcal N_{k+1} $, $ \alpha_p $ satisfy \cref{ass:stagecostproperty}, and $ \alpha_n $, $ \alpha_v $ satisfy \cref{lem:effectoflinoncosts}.
\end{ass}
\textcolor{\changes}{This assumption implies that the approximation error from using the stabilizing linear control law inside the terminal region must be negligible.} 
\cref{ass:nominalahnmpcnleffects} can be satisfied by finding a suitably large $N_\text{min}$ that is determined through simulations as explained in \cite{griffith2018robustly}.
\begin{thm}
	(ISpS stability of adaptive horizon multi-stage MPC)  
	Suppose \cref{ass:stagecostproperty} and \ref{ass:terminconditionsahmsnmpc} hold for \eqref{eq:AHscenarioNLPformulation}, the optimal value function $ V^{ahm} $ is an ISpS Lyapunov function and the resulting closed-loop system is ISpS stable.
\end{thm}  
\begin{proof}
	To show the ISpS property of the adaptive horizon multi-stage MPC, a Lyapunov function must exist that satisfies the three conditions in \eqref{eq:ISpSLyapunovFunction}.
	
	\emph{Lower bound:} Suppose \cref{ass:stagecostproperty} holds then there exists a lower bound \eqref{eq:ISpSlowerbound} on the optimal value function.
	\begin{equation}
		V_{N}^\mathrm{ahm}(x_k, \mathbf{d}^c) \geq \ell(x_k, \kappa_N(x_k), d_0^0) \geq  \alpha_p(|x_k|) \quad \forall x_k \in \mathbb X
	\end{equation}
	
	\emph{Upper bound}: The existence of an upper bound \eqref{eq:ISpSupperbound} is shown by using \cref{ass:basicstabilityassumptions}\ref{itm:basicstabilityassumptions1}, the boundedness property of the optimal cost of standard MPC recursively forward \citep{rawlings2017model}, and \cref{ass:terminconditionsahmsnmpc}\ref{itm:terminconditionsahmsnmpc3} as follows:
	\begin{subequations}
		\begin{align}
			V_{N}^\mathrm{ahm}(x_k, \mathbf{d}^c) &\leq V_{N}^\mathrm{ahm}(x_k, \mathbf{d}^0) + L_v|\Delta \mathbf{d}| \\
			&\leq \psi(x_k, d^0) + L_v|\Delta \mathbf{d}| \\
			&\leq \alpha_{q,\psi}(|x_k|) + c_1
		\end{align}
	\end{subequations} 
    for all $x_k \in \mathbb X_f$, where $ c_1 = L_v|\Delta \mathbf{d}| $ is a constant. 
	It also applies to all $x_k \in \mathbb{X}$ by induction because $\mathbb X_f \subseteq \mathbb X$, and $V_N^\mathrm{ahm}(x_k, \mathbf{d}^0)$ is continuous at the origin.
 
    \color{\changes}
    Note that the adaptive horizon multi-stage MPC control law may be discontinuous because it optimizes with different prediction horizons provided by \cref{alg:NLPsensalg} at each time step.
    However, our analysis considers the properties of the optimal cost function $V_{N_k}^{ahm}$ which may still be continuous even with a discontinuous control law, given that the input constraint set is continuous. \cite{rawlings2017model} 
    \color{black}
	
 \emph{Function descent:} The descent property \eqref{eq:ISpSdescent}, which is the main part of stability analysis, is shown for two possible cases of horizon update:

	1. \emph{A decreasing prediction horizon $ N_R \leq N_{k+1} < N_{k} $:}
	An approximate solution for \eqref{eq:AHscenarioNLPnominalequiv} at time $ t_{k+1} $ is found by shifting the optimal solution at time $ t_k $ as follows:
	\begin{equation} \label{eq:shiftedsolution}
		\hat{\nu}_{i|k+1}^0 = \nu_{i+1|k}^0, \quad \; i = 0, \dots, N_{k+1}-1,
	\end{equation}
	and the corresponding initialization for state variables:
	\begin{subequations} \label{eq:stateinitialization}
		\begin{align} 
			& \hat{z}_{0|k+1}^0 = z_{1|k}^0, \\
			& \hat{z}_{i+1|k+1}^0 = f(\hat{z}_{i|k}^0, \hat{\nu}_{i|k}^0, d_i^0)
		\end{align}
	\end{subequations}
	leads to a suboptimal cost such that the descent inequality becomes:
	\begin{subequations} \label{eq:case1nominaldescent}
		\begin{align}
			\begin{split}
			&V_{N_{k+1}}^\mathrm{ahm}(f(x_k, \kappa(x_k), d_0^0), \mathbf{d}^0)  - V_{N_k}^\mathrm{ahm}(x_k, \mathbf{d}^0) \\
			& \leq \sum_{c\in\mathbb{C}} \omega_c \Big (\psi(\hat {z}_{N_{k+1}|k+1}^0, d_{N_{k+1}-1}^0) + \sum_{i=0}^{N_{k+1}-1} \ell(\hat{z}_{i|k+1}^0, \hat{\nu}_{i|k+1}^0, d_i^0) \Big)\\
			 & - \sum_{c\in\mathbb{C}} \omega_c \Big (\psi(\nu_{N_k|k}^0, d_{N_k-1}^0) - \sum_{i=0}^{N_k-1} \ell(z_{i|k}^0, \nu_{i|k}^0, d_i^0) \Big)
			 \end{split} \\
			 \begin{split}
			 & = - \ell(x_k, \kappa(x_k), d_0^0) + \psi(\hat{z}_{N_{k+1}|k+1}^0, d_{N_{k+1}-1}^0)  \\
			 &  + \sum_{i=0}^{N_{k+1}-1} \big( \ell(\hat{z}_{i|k+1}^0, \hat{\nu}_{i|k+1}^0, d_i^0) - \ell(z_{i+1|k}^0, \nu_{i+1|k}^0, d_i^0) \big) \\
			 & - \psi(z_{N_k|k}^0, d_{N_k-1}^0) - \sum_{i=N_{k+1}+1}^{N_k-1} \ell(\hat{z}_{i|k}^0, \hat{\nu}_{i|k}^0, d_i^0)
			 \end{split} \\
			 \begin{split}
			 & = - \ell(x_k, \kappa(x_k), d_0^0)  + \psi(\hat{z}_{N_{k+1}|k+1}^0, d_{N_{k+1}-1}^0) \\
			 & - \psi(z_{N_k|k}^0, d_{N_k-1}^0) - \sum_{i=N_{k+1}+1}^{N_k-1} \ell(\hat{z}_{i|k}^0, \hat{\nu}_{i|k}^0, d_i^0)
			 \end{split} \\
			 \begin{split}
			 & = - \ell(x_k, \kappa(x_k), d_0^0)  \\
			 & + \psi(\hat{z}_{N_{k+1}|k+1}^0, d_{N_{k+1}-1}^0)  - V_{N_k - N_{k+1} - 1}^\mathrm{ahm}( z_{N_{k+1}+ 1|k}^0, \mathbf{d}^0)
			 \end{split}
		\end{align}
	In the presence of true uncertainty, the successor state is $ x_{k+1} = f(x_k, \kappa(x_k), d_k) $ where $ d_k $ is the true parameter realization, such that:
	\begin{align} \label{eq:case1robustdescent}
		\begin{split}
		& V_{N_{k+1}}^\mathrm{ahm} (x_{k+1}, \mathbf{d}^0) - V_{N_k}^\mathrm{ahm}(x_k, \mathbf{d}^0) \leq  - \ell(x_k, \kappa(x_k), d_0^0) \\
		& + \psi(\hat{z}_{N_{k+1}|k+1}^0, d_{N_{k+1}-1}^0)  - V_{N_k - N_{k+1} - 1}^\mathrm{ahm}(z_{N_{k+1}+ 1|k}^0, \mathbf{d}^0) \\
		& + L_k|d_k - d_i^0|
		\end{split}	
	\end{align} 
	\end{subequations}
	Then from \cref{eq:AHscenarioNLPformulation,eq:AHscenarioNLPnominalequiv} we have:
     \begin{subequations}
     	\begin{align} 
    		V_{N_k}^\mathrm{ahm}(x_k, \mathbf{d}^c) & = V_{N_k}^\mathrm{ahm}(x_k, \mathbf{d}^0) + \mathcal O(|\Delta \mathbf{d}|) \label{eq:nominalrobusterrk}, \: \text{and} \\
            V_{N_{k+1}}^\mathrm{ahm}(x_{k+1}, \mathbf{d}^c) & = V_{N_{k+1}}^\mathrm{ahm}(x_{k+1}, \mathbf{d}^0) + \mathcal O(|\Delta \mathbf{d}|) \label{eq:nominalrobusterrk+1}
    	\end{align} 
    \end{subequations}
	that can be combined with \eqref{eq:case1robustdescent} to form:
	\begin{subequations}
		\begin{align}
			\begin{split}
				& V_{N_{k+1}}^\mathrm{ahm}(x_{k+1}, \mathbf{d}^c)  - V_{N_k}^\mathrm{ahm}(x_k, \mathbf{d}^c)	\leq - \ell(x_k, \kappa(x_k), d_0^0) \\
				& + \psi(\hat{z}_{N_{k+1}|k+1}^0, d_{N_{k+1}-1}^0)  - V_{N_k - N_{k+1} - 1}^\mathrm{ahm}(z_{N_{k+1}+ 1|k}^0, \mathbf{d}^0) \\
				& + L_k|d_k - d_i^0| + \Big(V_{N_{k+1}}^\mathrm{ahm}(x_{k+1}, \mathbf{d}^c) - V_{N_{k+1}}^\mathrm{ahm}(x_{k+1}, \mathbf{d}^0)\Big) \\
				& + \Big( V_{N_k}^\mathrm{ahm}(x_k, \mathbf{d}^c) - V_{N_k}^\mathrm{ahm}(x_k, \mathbf{d}^0) \Big)
			\end{split} \\
			\begin{split}
				& \leq - \ell(x_k, \kappa(x_k), d_0^0) + \psi(\hat{z}_{N_{k+1}|k+1}^0, d_{N_{k+1}-1}^0) \\  
				& - V_{N_k - N_{k+1} - 1}^\mathrm{ahm}(z_{N_{k+1}+ 1|k}^0, \mathbf{d}^0) + L_k|d_k - d_i^0| + 2L_v|\Delta \mathbf{d}|
			\end{split} \\
				& \leq - \alpha_p(|x_k|) + \alpha_v(|z_{N_{k+1}+1|k}^0|) + L_k|d_k - d_i^0| + 2L_v|\Delta \mathbf{d}| \label{eq:ISpSdescentproof1c} \\
			& \leq -\alpha_3(|x_k|) + \sigma(|d_k - d_i^0|) + c_2 \label{eq:ISpSdescentproof1d}
		\end{align}
	\end{subequations}
	where $ c_2 = 2L_v|\Delta \mathbf{d}| \geq 0 $ and \eqref{eq:ISpSdescentproof1c} follows from \cref{lem:effectoflinoncosts} and \eqref{eq:ISpSdescentproof1d} follows from \eqref{eq:nominalahnmpceffects2}.
	 
2. \emph{An increasing prediction horizon $ N_{k+1} \geq N_{k} $:}
	Similar to the first case approximate the solution of \eqref{eq:AHscenarioNLPnominalequiv} at $ t_{k+1} $ by shifting: 
	\begin{equation} \label{eq:shiftedsolutioncase2}
			\hat {\nu}_{i|k+1}^0 = 
			\begin{cases} 
				\nu_{i+1|k}^0 & \quad \forall \; i = 0, \dots, N_k-2 \\ 
				- K_0 z_{i}^0 & \quad \forall \; i = N_k-1, \dots, N_{k+1}-1
			\end{cases} \\
	\end{equation}
	The initialization of the states is the same as in \eqref{eq:stateinitialization}.
	Then the descent inequality becomes:
	\begin{subequations} \label{eq:case2nominaldescent}
		\begin{align}
			& V_{N_{k+1}}^\mathrm{ahm}(f(x_k, \kappa(x_k), d_0^0), \mathbf{d}^0)  - V_{N_k}^\mathrm{ahm}(x_k, \mathbf{d}^0) \\
			\begin{split}
			 & \leq \sum_{c\in\mathbb{C}} \omega_c \Big( \psi(\hat {z}_{N_{k+1}|k+1}^0, d_{N_{k+1}-1}^0) + \sum_{i=0}^{N_{k+1}-1} \ell(\hat{z}_{i|k+1}^0, \hat{\nu}_{i|k+1}^0, d_l^0) \Big) \\
			 & - \sum_{c\in\mathbb{C}} \omega_c \Big( \psi(z_{N_k|k}^0, d_{N_k-1}^0) - \sum_{i=0}^{N_k-1} \ell(z_{i|k}^0, \nu_{i|k}^0, d_i^0) \Big)
			 \end{split} \\
			 \begin{split}
			 & = - \ell(x_k, \kappa(x_k), d_0^0)  \\
			 & + \sum_{i=0}^{N_k-2} \big( \ell(\hat{z}_{i|k+1}^0, \hat{\nu}_{i|k+1}^0, d_i^0) - \ell(z_{i+1|k}^0, \nu_{i+1|k}^0, d_i^0) \big) \\
			 & + \psi(\hat {z}_{N_{k+1}|k+1}^0, d_{N_{k+1}-1}^0) + \sum_{i=N_k-1}^{N_{k+1}-1} \ell(\hat{z}_{i|k}^0, \hat{\nu}_{i|k}^0, d_i^0) \\
			 & - \psi(z_{N_k|k}^0, d_{N_k-1}^0)
			 \end{split} \\
			 \begin{split}
			 & = - \ell(x_k, \kappa(x_k), d_0^0) + \psi(\hat {z}_{N_{k+1}|k+1}^0, d_{N_{k+1}-1}^0)  \\
			 & + \sum_{i=N_k-1}^{N_{k+1}-1} \ell(\hat{z}_{i|k+1}^0, \hat{\nu}_{i|k+1}^0, d_i^0) - \psi(z_{N_k|k}^0, d_{N_k-1}^0)
			 \end{split} \\
			 \begin{split}
			 & = - \ell(x_k, \kappa(x_k), d_0^0) + V_{N_k - N_{k+1} + 1}(z_{N_k|k}^0, \mathbf{d}^0)  \\
			 & - \psi(z_{N_k|k}^0, d_{N_k-1}^0)
			 \end{split}
		\end{align}
	In the presence of true parameter realization $ d_k $ then:
		\begin{align} \label{eq:case2robustdescent}
		\begin{split}
			& V_{N_{k+1}}^\mathrm{ahm} (x_{k+1}, \mathbf{d}^0) - V_{N_k}^\mathrm{ahm}(x_k, \mathbf{d}^0) \leq  - \ell(x_k, \kappa(x_k), d_0^0) \\
			& + V^{lqr}_{N_{k+1} - N_k + 1}(z_{N_k|k}^0, \mathbf{d}^0)  - \psi(z_{N_k|k}^0, d_{N_k-1}^0) \\
			& + L_k|d_k - d_i^0|
		\end{split}	
		\end{align}
	\end{subequations}
	Substituting \eqref{eq:nominalrobusterrk} and \eqref{eq:nominalrobusterrk+1} on \eqref{eq:case2robustdescent} we have:
	\begin{subequations}
	\begin{align}
		\begin{split}
			& V_{N_{k+1}}^\mathrm{ahm} (x_{k+1}, \mathbf{d}^c)  - V_{N_k}^\mathrm{ahm}(x_k, \mathbf{d}^c)	\leq - \ell(x_k, \kappa(x_k), d_0^0) \\
			& + V^{lqr}_{N_{k+1} - N_k + 1}(z_{N_k|k}^0, \mathbf{d}^0)  - \psi(z_{N_k|k}^0, d_{N_k-1}^0) \\
			& + L_k|d_k - d_i^0| +  2L_v|\Delta \mathbf{d}| \\
		\end{split}\\
		\begin{split}
			& \leq - \ell(x_k, \kappa(x_k), d_0^0) + V^{lqr}_{N_{k+1} - N_k + 1}(z_{N_k|k}^0, \mathbf{d}^0) \\  
			&- \psi(z_{N_k|k}^0, d_{N_k-1}^0) + L_k|d_k - d_i^0| + 2L_v|\Delta \mathbf{d}| 
		\end{split}\\
			& \leq - \alpha_p(|x_k|) + \alpha_n(|z_{N_{k|k}}^0|) + L_k|d_k - d_i^0| + 2L_v|\Delta \mathbf{d}|  \label{eq:ISpSdescentproof2c}\\
			& \leq -\alpha_3(|x_k|) + \sigma(|d_k - d_i^0|) + c_2 \label{eq:ISpSdescentproof2d}
	\end{align}
	\end{subequations}
	where $ c_2 = 2L_v|\Delta \mathbf{d}| \geq 0 $, and \eqref{eq:ISpSdescentproof2c} follows from \cref{lem:effectoflinoncosts} and \eqref{eq:ISpSdescentproof2d} follows from \eqref{eq:nominalahnmpceffects1}.
	
Thus $ V_{N_k}^\mathrm{ahm} $ satisfies the descent property \eqref{eq:ISpSdescent} for all the possible horizon length updates. 
Hence the adaptive horizon multi-stage MPC is ISpS stable.
\end{proof}

\section{Numerical examples} \label{sec:casestudies}
\color{\changes}
\subsection{Example 1 - Spring damped mass system}
This is a simple example adapted from \cite{raimondo2009min} used to demonstrate the performance of an adaptive horizon multi-stage MPC with a fully branched scenario tree, and therefore suitable terminal conditions that guarantee recursive feasibility and ISpS stability (see \cref{sec:recursivefeasibilityfullybranched}).
The system model is written as:
\begin{equation*}
    \dot x_1 = x_2, \quad \dot x_2 = -\frac{k_0}{m} {\rm e}^{-x_1} x_1 - \frac{h_d}{m} x_2 + \frac{u}{m}
\end{equation*}
where the state variable vector $x = [x_1, \: x_2]^\top$ includes the displacement and the velocity of the mass $m = 1$ $\kilogram$, respectively.
$k = k_0{\rm e}^{-x_1}$ is the elastic constant of the spring and $k_0 = 0.33$ N$\per\meter$.
The damping factor $h_d$ is uncertain and can have three different possible values $h_d^r \: = \: \{1.0,2.0,4.0\}$ N$\second\per\meter$ with equal probabilities.
The damping factor is assumed to vary unpredictably between sampling intervals of $\Delta t$ = 0.4 \second.

The objective is to control the mass to its equilibrium position $x_1 = 0$ using an external force $u$.
The stage cost is given by $\ell = x^\top Q x + u^\top R u$ and $h_d$ varies randomly at each time step with the previously defined probabilities.

\subsection{Approximation of terminal ingredients}
The LQR method explained in \cref{sec:terminalconditions} is applied to find suitable $\psi$ and $\mathbb X_f$.
The tuning matrices used are $Q = \text{diag}([30, \; 20])$ and $R=1$.
The terminal control law must be common for all scenarios due to full branching (see \cref{sec:terminalingredientsfullybranched}).
Therefore, $K = [1.7409 \; 2.0959]$ is chosen with $\kappa_f(x) = -Kx$ for all $x \in \mathbb X_f$.
Different values of $P_r$ are computed using the Lyapunov equation to obtain the respective terminal costs for each realization.
\begin{figure}
	\centering
	\includegraphics[scale=0.8]{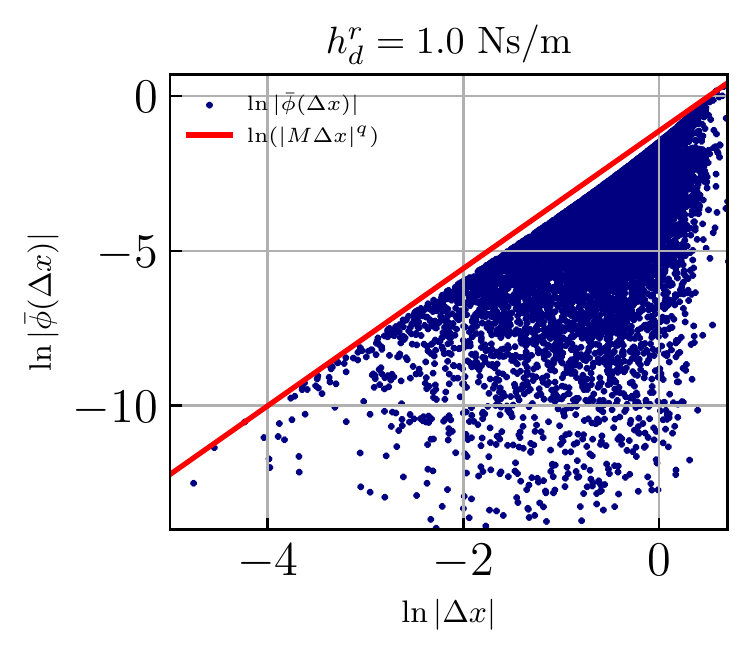}
	\caption{\color{\changes}Spring-damper-mass - plots of linearization error against $ |\Delta x| $ for $ 10000 $ simulations showing the bound (red line).}
	\label{fig:sdmlinerrplt}
\end{figure}
The terminal radii are determined by performing $10^5$ one-step simulations of the closed-loop system and estimating the linearization error bound.
The linearization error of this system is independent of $h_d$, so the simulations were performed for only one of the disturbance realizations.
After obtaining bound parameters $M_r = 0.3235$ and $q_r = 2.2176$, the terminal radii were found to be $c_f^r \approx \{0.8032, \:
 0.8922, \:
 0.6690\}$.
 \subsection{Simulation results}
The simulations are done using \texttt{JuMP v.0.21.10} \citep{DunningHuchetteLubin2017} as the NLP modeler in a \texttt{Julia} \citep{doi:10.1137/141000671} environment. 
The NLP solver used is \texttt{IPOPT 3.13.4} \citep{wachter2006implementation}, and the linear solver is \texttt{HSL-MA57} \citep{hsl2007collection} on a 2.6 GHz Intel Core-i7 with 16 GB memory.

An initial prediction horizon of $N_0 = 8$ is chosen with $N_\text{min} = 2$.
The scenario tree is fully branched at each MPC iteration making a total of 6561 scenarios at the first iteration.
\cref{fig:sdmresults} shows two sets of closed-loop simulations of the adaptive horizon multi-stage MPC that are done with initial conditions $x_0 = [-4,\: 4]^\top$ $\meter$ (blue line) and $ [5.3,\: 2]^\top$ $\meter$ (black line).
The state bounds are $x_1 \in [-10,\: 10]$ $ \meter$  and $x_2 \in [-4,\:10]$ \meter\per\second and the control input bounds are $u \leq |6.0 \: N|$.
 \begin{figure}
	\centering
	\includegraphics[scale=1.0]{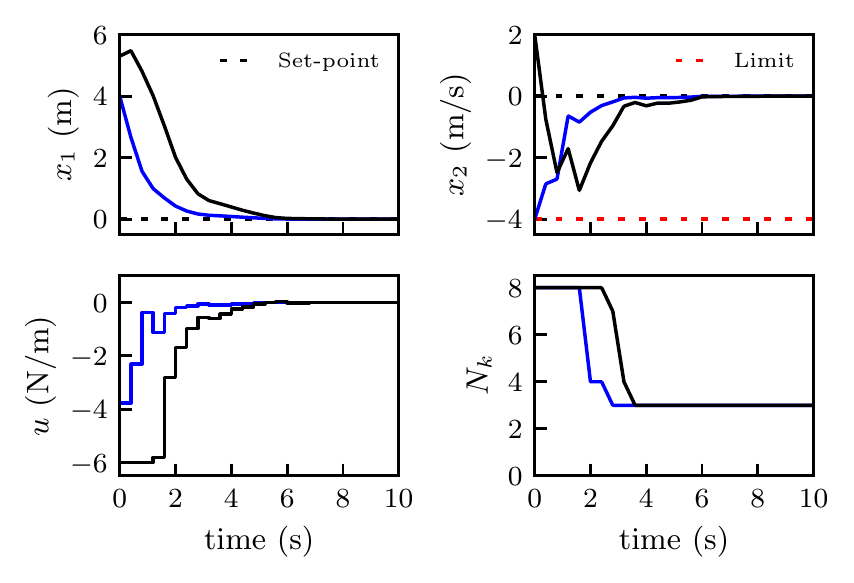}
	\caption{\color{\changes}Spring-damper-mass - simulation results from two initial conditions $x_0 = [-4,\: 4]$ (blue lines) and $ x_0 = [5.3,\: 2]^\top$ (black lines) showing the control performance of a fully branched adaptive horizon multi-stage MPC.}
	\label{fig:sdmresults}
\end{figure}
The system is controlled to its setpoint and the horizon length $N_k$ is reduced as the setpoint is approached.
Eventually, the $N_k$ is reduced to 3 from 8 cutting the number of scenarios from 6561 to 27.
Therefore, this example also demonstrates the advantage of the adaptive horizon method for computational cost savings in multi-stage MPC with a fully branched scenario tree.

\color{black}
\subsection{Example 2 - Cooled CSTR}
\textcolor{\changes}{The rest of the numerical examples consider a robust horizon (i.e. no full branching), therefore recursive feasibility is not necessarily guaranteed (see \cref{sec:recursivefeasibilityrobusthorizon}).
However, the aim is to compare the performances of the adaptive horizon multi-stage MPC (varying horizon lengths) with fixed horizon multi-stage MPC and to demonstrate the computational cost savings as a result of the adaptive horizon algorithm.}

Consider the control of a CSTR with a cooling jacket example from \cite{klatt1998gain}. 
The dynamics of the cooled CSTR are given by,
\begin{align*}
    \dot{c_A} & = F(c_{A,0} - c_A) - k_1c_A - k_3c_A^2 \\
    \dot{c_B} & = -Fc_B + k_1c_A - k_2c_B \\
    \begin{split}
    \dot{T_R} & = F(T_{\text{in}} - T_R) + \frac{k_WA_R}{\rho c_p V_R}(T_J - T_R)\\
    & \qquad - \frac{k_1c_A\Delta H_{AB} + k_2c_B\Delta H_{BC} + k_3c_A^2\Delta H_{AD}}{\rho c_p} 		
    \end{split}\\
    \dot{T_J} & = \frac{1}{m_Jc_{p,J}}(\dot{Q_J} + k_WA_R(T_R - T_J))	  
\end{align*}
where reaction rates $ k_i $ follow the Arrhenius law, $ k_i = A_i\exp{\Big(\frac{-E_{i}}{RT_R}\Big)} $.
\begin{table}
	\caption{CSTR - Bounds on states and inputs}
	\centering
	\begin{tabular}{|c|c|c|c|c|}
		\hline
		Variable & Initial condition & Minimum & Maximum & Unit \\
		\hline
		$c_A$ & $0.8$ & $0.1$ & $5.0$ & $\mole\per\litre$ \\
		$c_B$ & $0.5$ & $0.1$ & $5.9$ & $\mole\per\litre$ \\ 
		$T_R$ & $134.14$ & $50$ & $140$ & $\degreecelsius$ \\ 
		$T_J$ & $134.0$ & $50$ & $180$ & $\degreecelsius$ \\ 
		$F$ & $18.83$ &$0.0$ & $35$ & $\per\hour$ \\
		$\dot{Q_J}$& $-4495.7$ & $-8500$ & $0$ & $\kilo\joule\per\hour$ \\
		\hline
	\end{tabular}
	\label{tab:cooledcstrbounds}
\end{table}

The state variable vector $ x = [c_A, \, c_B,\, T_R,\, T_J]^\top $ consists of the concentrations of $ A $ and $ B $, reactor and coolant temperatures, respectively. 
The control inputs $ u $, are inlet flow per reactor volume $ F = V_\text{in}/V_R $, and cooling rate $ \dot{Q_J} $.

The control objective is to regulate $ c_B $ at a desired setpoint. 
The system operates at two setpoints: $ c_B^\text{set} = 0.5 \, \mole\per\ell $, and $ c_B^\text{set} = 0.7 \, \mole\per\ell $.
The activation energy $ E_{3} =  8560 R \pm 2.5\% $ $\kelvin$ is the only uncertain parameter in the system.
The stage cost is a setpoint tracking squared error of $ c_B $ plus control movement penalization terms $ \Delta F_i = F_{i}-F_{i-1} $ and $ \Delta \dot{Q_J}_k = \dot{Q_J}_{i}-\dot{Q_J}_{i-1} $ given by $ \ell_k = ({c_B}_i - c_B^\text{set})^2 + {r_\Delta}_1\Delta F_i^2 + {r_\Delta}_2\Delta \dot{Q_J}_i^2 $.
where the control penalties are $ {r_\Delta}_1 = 10^{-5} $ and $ {r_\Delta}_2 = 10^{-7} $.
Regularization terms are added to impose implicit references on the remaining states and strong convexity \citep{jaschke2014fast}. 

The initial prediction horizon is $ N_0  = 40 $. 
Terminal constraints are included for all the scenarios.
The simulations are performed for $ N_R = 1,\, 2 $.
We sample three parameter realizations, such that when $N_R = 2$ there are 9 scenarios as in \cref{fig:robusthorizonscenariotree}.

\subsubsection{Offline approximation of terminal ingredientd}
To design a stabilizing LQR for each  parametric realization, we use $ Q = I_4 $, and $ R = \text{diag}([10^{-3} \, 10^{-4}]) $. 
As outlined before, $ 10^5 $ one-step simulations are performed offline from random initial states to compute linearization error for each parametric realization and each operating point.
\cref{fig:cstrlinerrplt} shows the plots of the linearization error against $ |\Delta x| $ for each parametric realization and each setpoint.
\begin{figure*}
	\centering
	\includegraphics[width=0.9\textwidth]{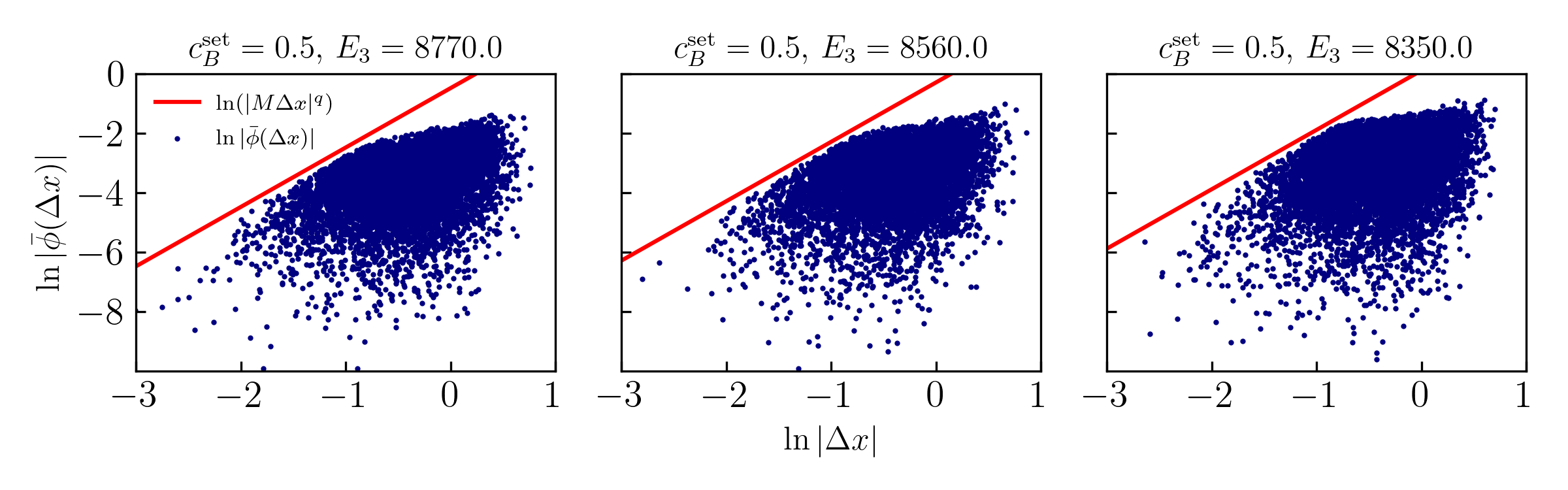}
	\includegraphics[width=0.9\textwidth]{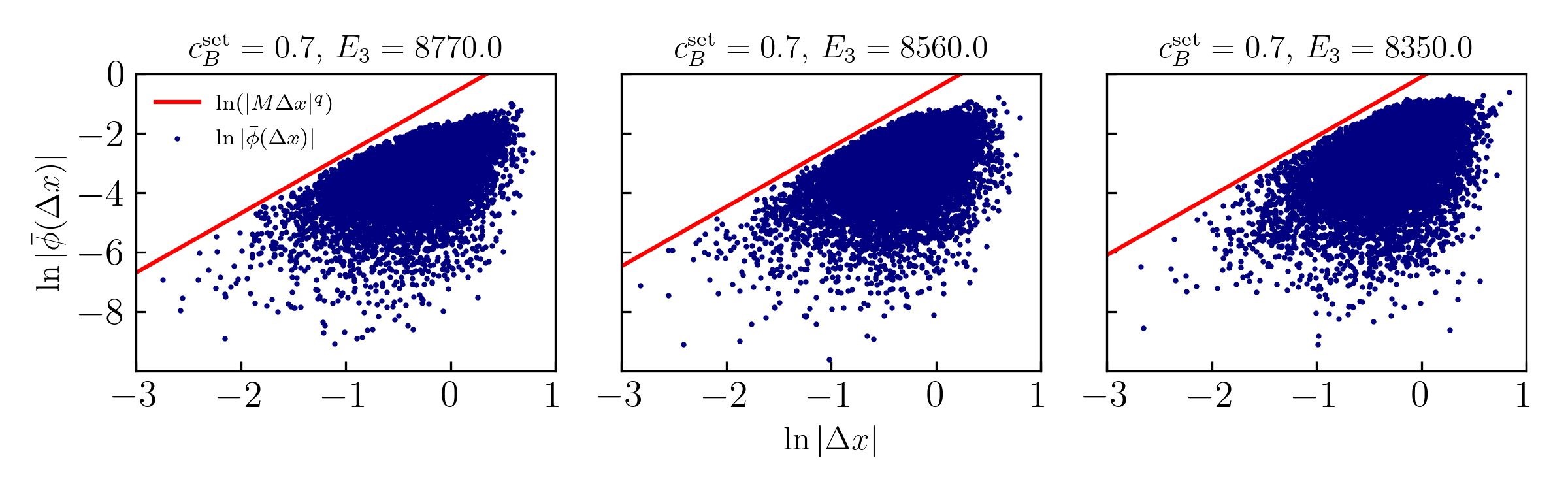}
	\caption{CSTR - plots of linearization error against $ |\Delta x| $ for $ 10000 $ simulations for each parametric realization of $E_3$, at each setpoint $c_B^\text{set} = 0.5$ mol/$\ell$ and 0.7 mol/$\ell$ showing their bounds (red lines).}
	\label{fig:cstrlinerrplt}
\end{figure*}
All the plots show an upper bound on the linearization error (red solid lines in \cref{fig:cstrlinerrplt}) from which $M_r$ and $q_r$ values are fitted.
The terminal region radii for each parametric realization are obtained by evaluating \eqref{eq:cf} as shown in \cref{tab:cstrlinerrorbounds}.
\begin{table}
	\caption{CSTR - linearization error bounds and terminal radii}
	\centering
	\begin{tabular}{|c|c|c|c|c|c|c|}
		\hline
		 \multirow{3}{*}{$ E_3 $ (\kelvin)} & \multicolumn{3}{c}{$ c_B^\text{set} = 0.5$ mol/$ \ell $}  & \multicolumn{3}{|c|}{$ c_B^\text{set} = 0.7$ mol/$ \ell $} \\
		\cline{2-7}
		& $ M_r $ & $ q_r $ & $ c_f^r $ & $ M_r $ & $ q_r $ & $ c_f^r $ \\ 
		\hline
		$8774$ & $0.62$ & $2.0$ & $0.1429$ & $0.50$ & $ 2.0 $ & $ 0.1718$ \\
		$8560$ & $0.75$ & $2.0$ & $0.1159$ & $0.62$ & $ 2.0 $ & $ 0.1337$ \\ 
		$8346$ & $1.12$ & $2.0$ & $0.0747$ & $0.90$ & $ 2.0 $ & $ 0.0846$ \\
		\hline
	\end{tabular}
	\label{tab:cstrlinerrorbounds}
\end{table}

\subsubsection{Simulation results}
A sample time of $ 18 \, \second $ is used, and the process is simulated for $ 0.6 \, \hour $.
At the beginning $ c_B^\text{set} = 0.5 \, \mole\per\ell $ and it changes to $ c_B^\text{set} = 0.7 \, \mole\per\ell $ at $ t_k = 0.3 \, \hour $.   
The uncertain parameter $ E_{3} $ is a random sequence of the sampled realizations.
The control performance of the adaptive horizon multi-stage MPC is compared with the fixed horizon multi-stage MPC in this system. 
The value of $ N_\text{min} = 5 $ in the adaptive horizon multi-stage algorithm.
\cref{fig:statesandinputs1,fig:statesandinputs2} show closed-loop simulation results when $N_R = 1, \, 2$, respectively.
\begin{figure}
	\centering
	\includegraphics[scale=1.0]{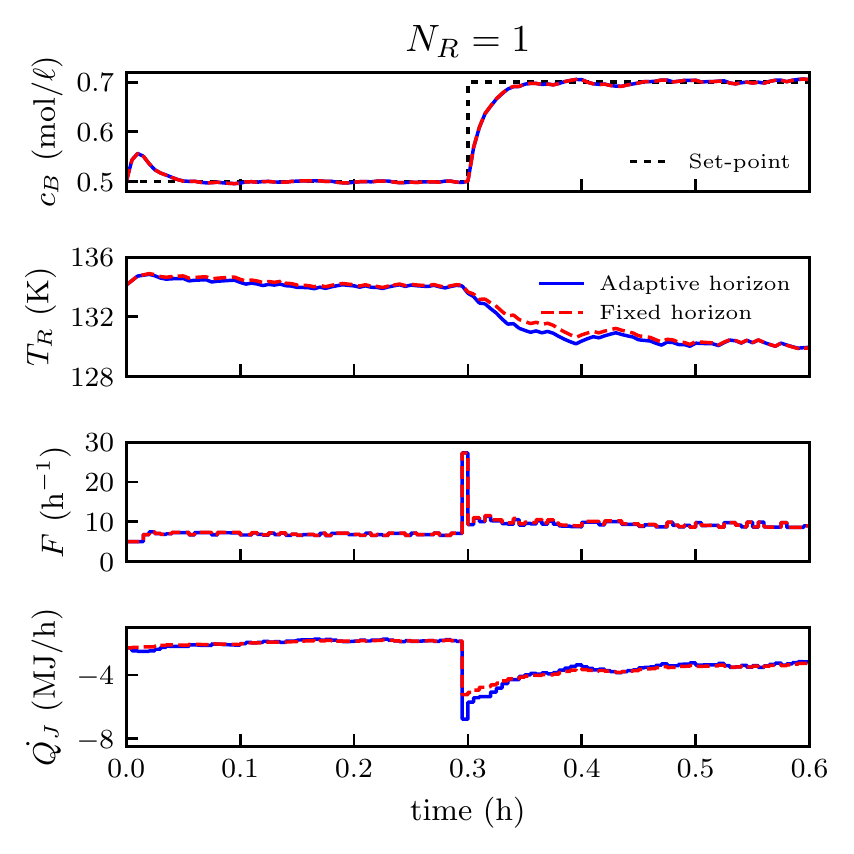}
	\caption{CSTR - Simulation results comparing the adaptive horizon multi-stage MPC with the fixed horizon multi-stage MPC when robust horizon $N_R=1$.} 
	\label{fig:statesandinputs1}
\end{figure}
\begin{figure}
	\centering
	\includegraphics[scale=1.0]{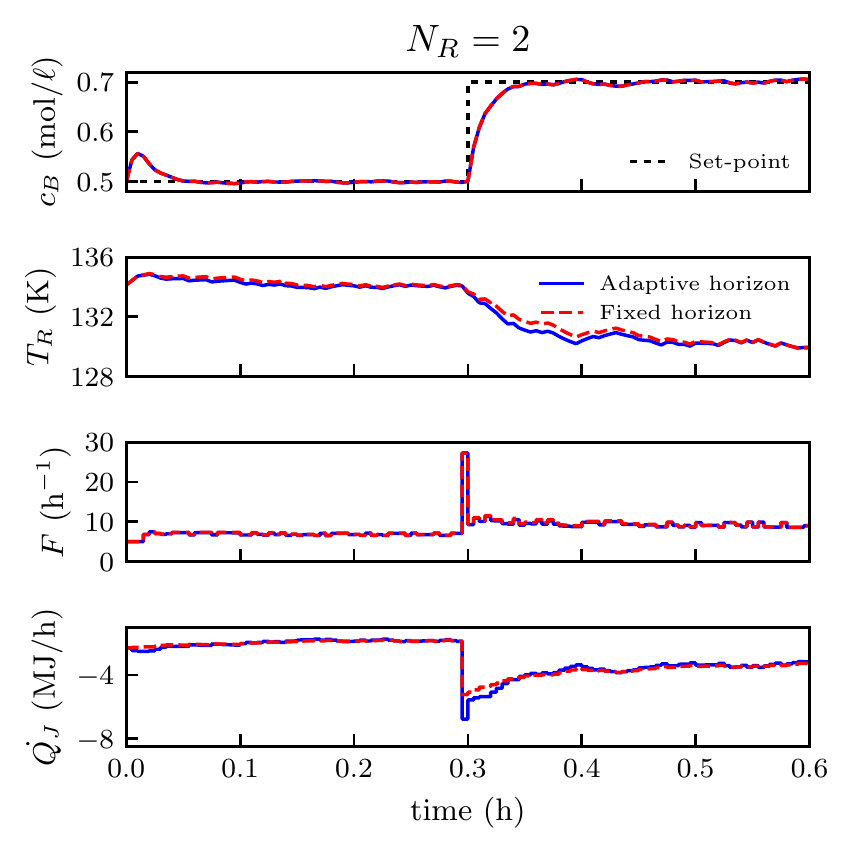}
	\caption{CSTR - Simulation results comparing the adaptive horizon multi-stage MPC with the fixed horizon multi-stage MPC when robust horizon $N_R=2$.}
	\label{fig:statesandinputs2}
\end{figure}
The optimal input sequences and state trajectories for both controllers are nearly identical.
The two controllers show similar performance in tracking changes in setpoint $ c_B^\text{set} $.
This implies that the adaptive horizon algorithm does not affect the tracking performance of the multi-stage MPC.

The prediction horizon and total computation times per iteration for the two controllers when $ N_R = 1,\,2$ are plotted in \cref{fig:computationtime1,fig:computationtime2}, respectively.
\begin{figure}
	\centering
	\includegraphics[scale=1.0]{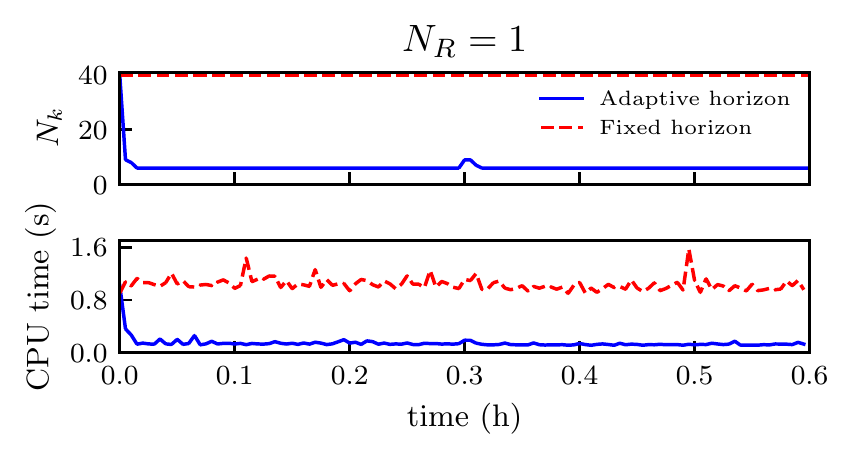}
	\caption{CSTR - Simulation results comparing prediction horizons and computation times per iteration when robust horizon $ N_R = 1 $.}
	\label{fig:computationtime1}
\end{figure}
\begin{figure}
	\centering
	\includegraphics[scale=1.0]{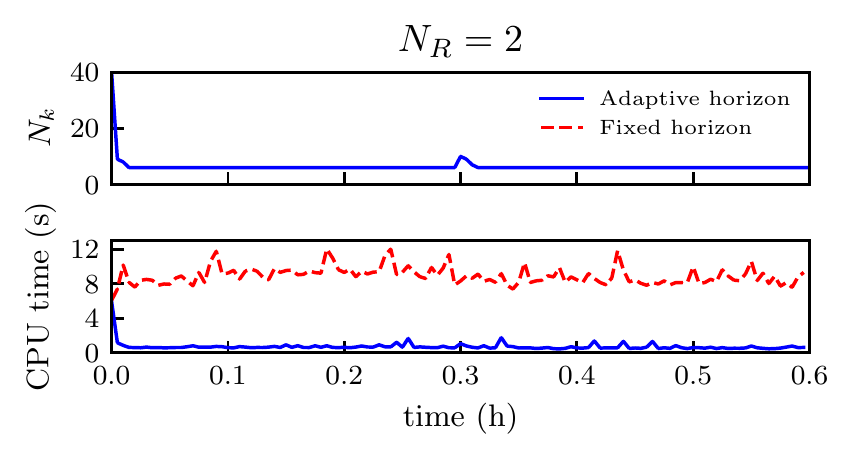}
	\caption{CSTR - Simulation results comparing prediction horizons and computation times per iteration when robust horizon $ N_R = 2 $.}
	\label{fig:computationtime2}
\end{figure}
The adaptive horizon multi-stage MPC reduces the prediction horizon significantly, thus saving the computation time needed in each multi-stage MPC iteration.
There is an $ 86\% $ savings in CPU time on average from fixed horizon multi-stage MPC with $ N_R=1 $ by using the adaptive horizon multi-stage MPC.
When $ N_R=2 $, there is also a $ 92\% $ savings in CPU time on average.
The regulatory performances of the two controllers are nearly identical but the prediction horizon and computation time are reduced significantly for the adaptive horizon multi-stage MPC.

\subsection{Example 3 - Quad-tank system}
The second example is from \cite{raff2006nonlinear}, on the control of a quad-tank system with four tanks. 
The liquid levels in the four tanks are described by the following set of differential equations:
\begin{align*}
    \dot{x_1} &= -\frac{a_1}{A_1}\sqrt{2gx_1} + \frac{a_3}{A_1}\sqrt{2gx_3} + \frac{\gamma_1}{A_1}u_1 \\
    \dot{x_2} &= -\frac{a_2}{A_2}\sqrt{2gx_2} + \frac{a_4}{A_2}\sqrt{2gx_4} + \frac{\gamma_2}{A_2}u_2 \\
    \dot{x_3} &= -\frac{a_3}{A_3}\sqrt{2gx_3} + \frac{1-\gamma_2}{A_3}u_2 \\
    \dot{x_4} &= -\frac{a_4}{A_4}\sqrt{2gx_4} + \frac{1-\gamma_1}{A_4}u_1 
\end{align*}
where $x_i$ represents the liquid level in tank $i$, $u_i$ is the flow rate of pump $i$.
The system has four states: $ x = [x_1, \, x_2, \, x_3,\, x_4]^\top$ and two control inputs: $ u = [u_1,\, u_2]^\top $.
$ A_i $ and $ a_i $ are the cross sectional areas of the tank $i$ and its outlet, respectively.
The valve coefficients $\gamma_1$ and $\gamma_2$ are the uncertain parameters of this system.
The values of $\gamma_1, \gamma_2 = 0.4 \pm 0.05$ and the system bounds are shown in \cref{tab:quadtankbounds}.
\begin{table}
	\caption{Quad-tank - Bounds on states and inputs}
	\centering
	\begin{tabular}{|c|c|c|c|}
		\hline
		Variable & Minimum & Maximum & Unit \\
		\hline
		$x_1$ & $7.5$ & $28.0$ & $\centi\meter$ \\
		$x_2$ & $7.5$ & $28.0$ & $\centi\meter$ \\ 
		$x_3$ & $14.2$ & $28.0$ & $\centi\meter$ \\ 
		$x_4$ & $4.5$ & $21.3$ & $\centi\meter$ \\ 
		$u_1$ & $0.0$ & $60.0$ & $\milli\ell\per\second$ \\
		$u_2$ & $0.0$ & $60.0$ & $\milli\ell\per\second$ \\
		\hline
	\end{tabular}
	\label{tab:quadtankbounds}
\end{table}

The controller has a reference tracking objective, regulating the two levels in the lower tanks (tanks 1 and 2).
The setpoints are $x_1^\text{set} = x_2^\text{set} = 14 \, \centi\meter$.
We introduce predefined pulse changes in the state values at specific iterations to reset reference tracking as shown in \cref{tab:quadtankpulse}.
\begin{table}
	\caption{Pulse changes to state variables applied to the quad-tank system}
	\centering
	\begin{tabular}{|c|c|c|c|c|}
		\hline
		$k$ & $x_1$ & $x_2$ & $x_3$ & $x_4$ \\
		\hline
		$0$ & 28 $\centi\meter$ & 28 $\centi\meter$ & 14.2 $\centi\meter$ & 21.3 $\centi\meter$\\
		$50$ & 28 $\centi\meter$ & 14 $\centi\meter$ & 28 $\centi\meter$ & 21.3 $\centi\meter$\\ 
		$100$ & 28 $\centi\meter$ & 14 $\centi\meter$ & 14.2 $\centi\meter$ & 21.3 $\centi\meter$\\ 
		\hline
	\end{tabular}
	\label{tab:quadtankpulse}
\end{table}

The closed-loop multi-stage MPC simulations are run for 150 iterations with a sample time of $10 \, \second$.
The stage cost function is given by: $\ell_k = ({x_1}_i - x_1^\text{set})^2 + ({x_2}_i - x_2^\text{set})^2 + r_{\Delta}(\Delta {u_1}_i^2 + \Delta {u_2}_i^2)$,
where $ \Delta {u_1}_i = {u_1}_i - {u_1}_{i-1} $ and $ \Delta {u_2}_i = {u_2}_i - {u_2}_{i-1} $ are the control movement terms for the pump flow-rates that are penalized in the objective function with the penalty parameter $ r_{\Delta} = 0.01 $.

\subsubsection{Offline approximation of terminal ingredients}
LQR controller tuning $ Q = 1.5I_4 $ and $ R = I_2 $ are selected for the stabilizing controller for each parameter realization of $ \gamma_1 $ and $ \gamma_2 $. 
As done previously, $10^5$ one-step offline simulations from randomly sampled initial state values are performed for each parameter realization to compute the linearization error.
The set $\mathbb{D}$ is obtained via grid-based sampling including the nominal and extreme parameter values only.
A total of 9 parameter realizations are considered for the two uncertain parameters.
\begin{figure*}
	\centering
	\includegraphics[width=0.9\textwidth]{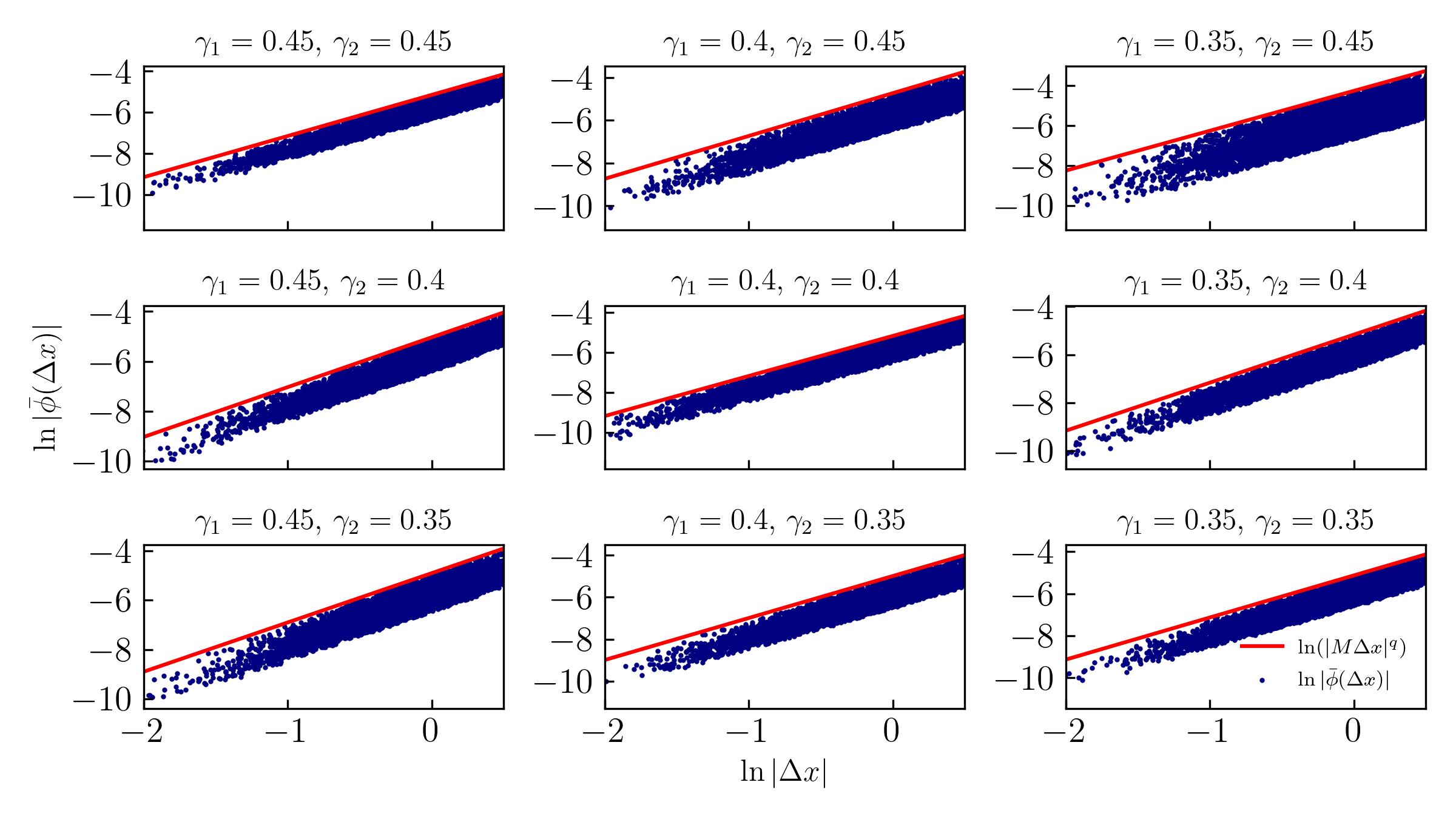}
	\caption{Quad-tank - matrix of plots of linearization error against $ |\Delta x| $ of $ 10,000 $ simulations for every uncertain parameter realization.}
	\label{fig:quadtanklinerrplt}
\end{figure*}
\cref{fig:quadtanklinerrplt} is a matrix of linearization error plots from the one-step simulations with estimated upper bounds (red lines) for each uncertain parameter realization.
The terminal radii for each parameter realization are computed using \eqref{eq:cf}.
The estimated linearization error bounds and terminal radii for each parameter realization are shown in \cref{tab:quadtanklinerrorbounds}.
\begin{table}
	\caption{Quad-tank - linearization error bounds and terminal radii}
	\centering
	\begin{tabular}{|c|c|c|c|}
		\hline
	 	$[\gamma_1; \, \gamma_2 ]$& $ M_r $ & $ q_r $ & $ c_f^r $ \\ 
		\hline
		$ [0.45; \, 0.45] $ & $0.0058$ & $2.0$ & $30.43$ \\
		$ [0.40; \, 0.45] $ & $0.0088$ & $2.0$ & $19.99$ \\ 
		$ [0.35; \, 0.45] $ & $0.0142$ & $2.0$ & $12.12$ \\
		$ [0.45; \, 0.40] $ & $0.0065$ & $2.0$ & $27.04$ \\
		$ [0.40; \, 0.40] $ & $0.0056$ & $2.0$ & $31.77$ \\
		$ [0.35; \, 0.40] $ & $0.0057$ & $2.0$ & $31.57$ \\
		$ [0.45; \, 0.35] $ & $0.0074$ & $2.0$ & $23.69$ \\
		$ [0.40; \, 0.35] $ & $0.0068$ & $2.0$ & $26.63$ \\
		$ [0.35; \, 0.35] $ & $0.0059$ & $2.0$ & $30.97$ \\
		\hline
	\end{tabular}
	\label{tab:quadtanklinerrorbounds}
\end{table}

\subsubsection{Simulation results}
The parameters $\gamma_1$ and $\gamma_2$ have an uncertainty range of $\pm0.05$ about their nominal value.
Similar to Example 2, the simulations are performed with a parameter sequence randomly sampled from \{max, nom, min\} values.
Again, there are two simulation sets comparing the control performances of the adaptive horizon and fixed horizon multi-stage MPC when $ N_R = 1,\,2$.
Simulation results of the quad-tank system for the two controllers when $ N_R = 1,\,2 $ are plotted in \cref{fig:quadtankstatesandinputs1,fig:quadtankstatesandinputs2}, respectively.
\begin{figure}
	\centering
	\includegraphics[scale=1.0]{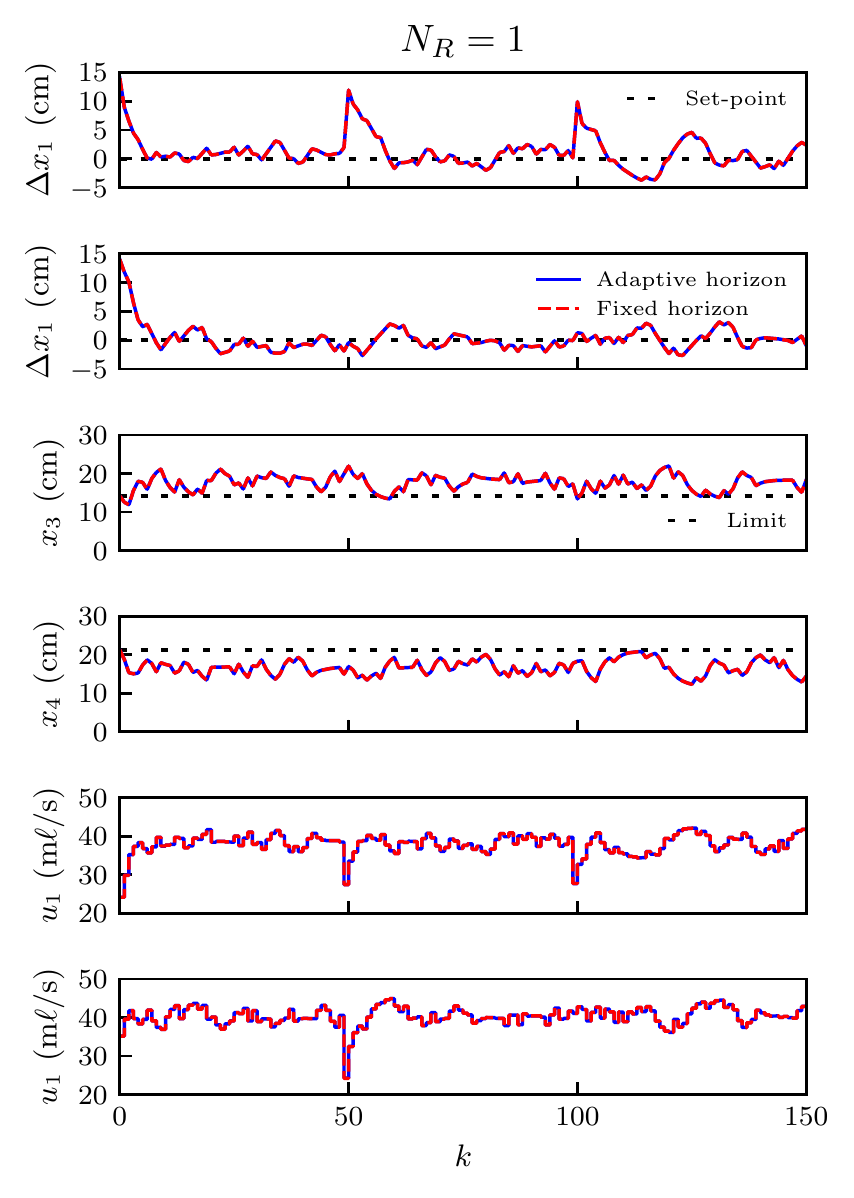}
	\caption{Quad-tank - Simulation results comparing fixed horizon multi-stage and adaptive horizon multi-stage MPC with $ N_R=1 $.}
	\label{fig:quadtankstatesandinputs1}
\end{figure}
\begin{figure}
	\centering
	\includegraphics[scale=1.0]{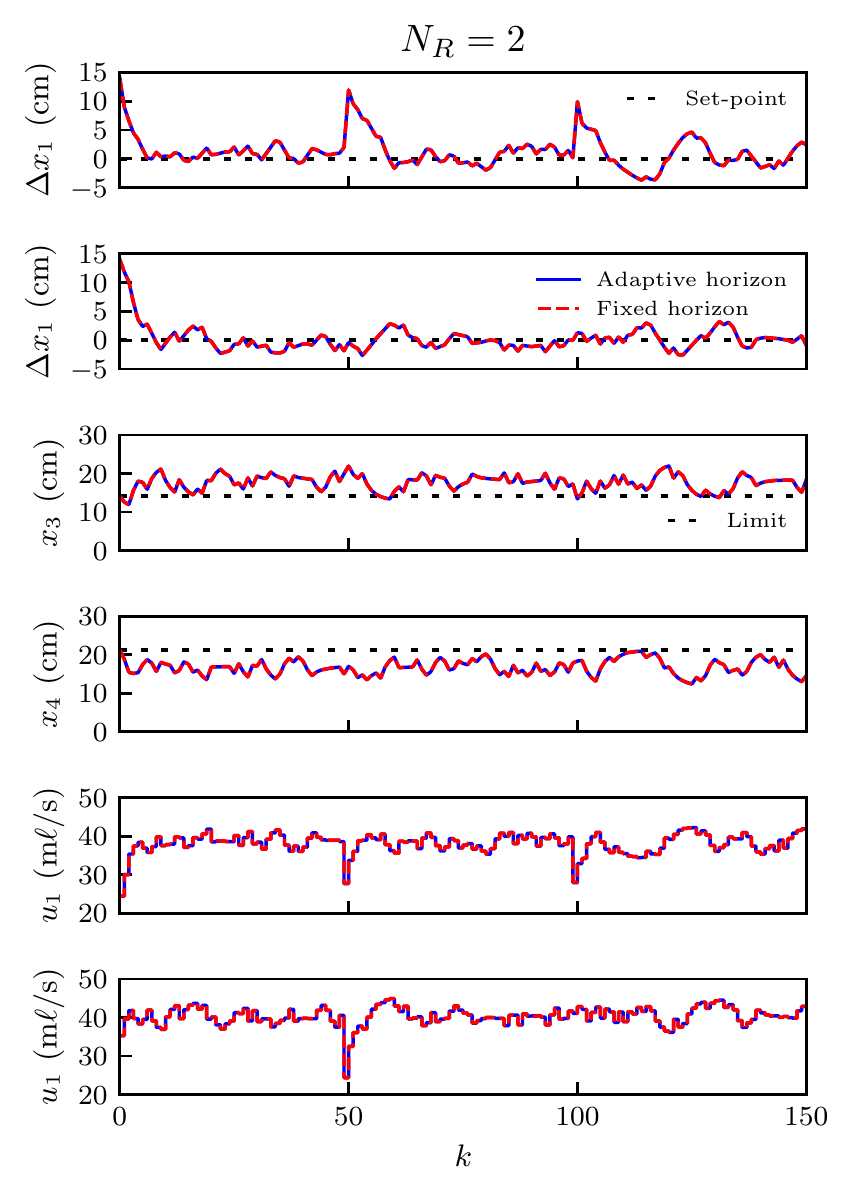}
	\caption{Quad-tank - Simulation results comparing fixed horizon multi-stage and adaptive horizon multi-stage MPC with $ N_R=2 $.}
	\label{fig:quadtankstatesandinputs2}
\end{figure}
It is seen from the levels $x_1$, $x_2$ that the setpoint tracking performance of the two controllers for both cases of $ N_R=1,\,2 $ are again nearly identical.
There are also no significant constraint violations in the bounds of $x_3$, $x_4$ for both controllers.
Therefore, there are no performance losses in both reference tracking and robustness of the multi-stage MPC by including the horizon update algorithm.

The prediction horizon and computation time per iteration for both controllers when $ N_R = 1,\,2 $ are plotted in \cref{fig:quadtankcomputationtime1,fig:quadtankcomputationtime2}, respectively.
\begin{figure}
	\centering
	\includegraphics[scale=1.0]{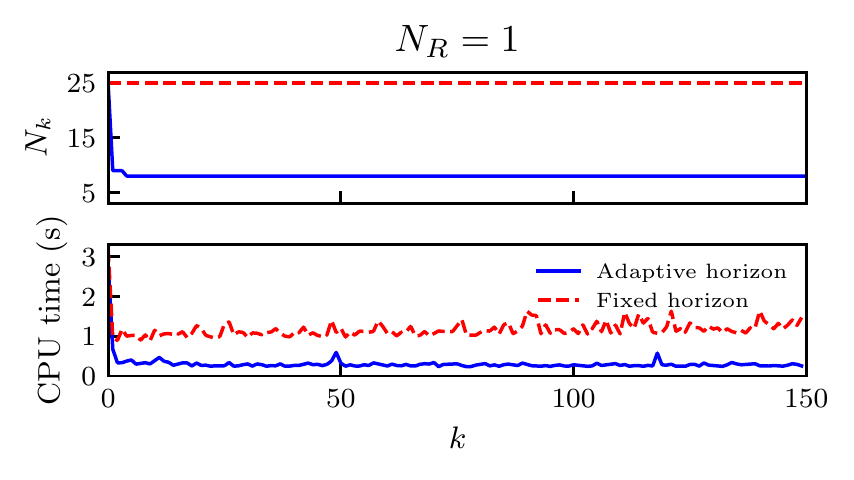}
	\caption{Quad-tank - plot comparing prediction horizons and computation times at each iteration for ideal multi-stage and adaptive horizon multi-stage MPC with $ N_R=1 $.}
	\label{fig:quadtankcomputationtime1}
\end{figure}
\begin{figure}
	\centering
	\includegraphics[scale=1.0]{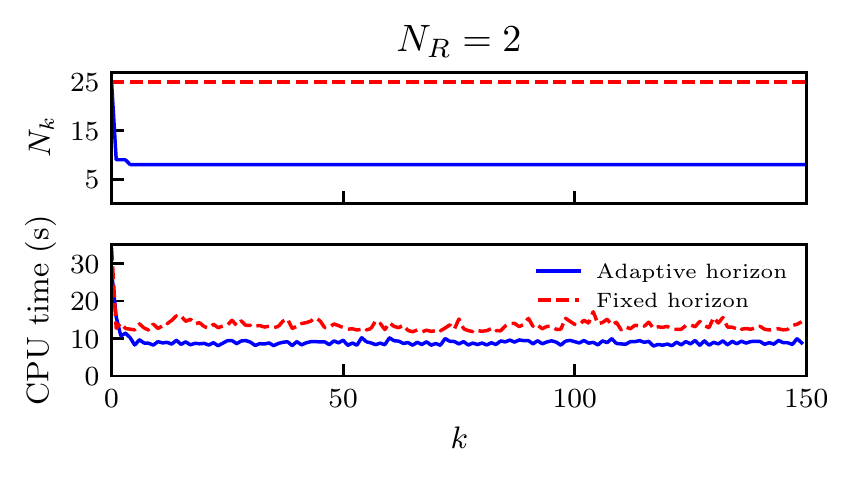}
	\caption{Quad-tank - plot comparing prediction horizons and computation times at each iteration for ideal multi-stage and adaptive horizon multi-stage MPC with $ N_R=2 $.}
	\label{fig:quadtankcomputationtime2}
\end{figure}
There is a reduction in the prediction horizon as the process approaches its setpoint.
The horizon update algorithm is unaffected by the pulse changes in the tank levels.
This is probably because the terminal radii values (see \cref{tab:quadtanklinerrorbounds}) are of the same order as the tank levels.
Therefore the pulse changes do not displace the system far enough from the terminal region to cause a significant horizon increase.
The average CPU time of the fixed horizon multi-stage MPC was significantly reduced by $ 74\% $ with an adaptive horizon when $ N_R = 1 $.
Similarly, the CPU time reduction is $ 33\% $ when $N_R=2$.

\section{Conclusion} \label{sec:conclusion}
\color{\changes}
This paper adopts the existing adaptive horizon algorithm and extends it for multi-stage MPC to improve much-needed computational efficiency.
First, the terminal costs and regions are computed independently for each uncertain parameter realization about their optimal steady states.
We assume that a common terminal region exists for all uncertain parameter realizations, and the prediction horizon length is updated such that it is always reached.

The adaptive horizon framework is shown to be recursively feasible when the scenario tree is fully branched. 
Implementation of soft constraints may be used to avoid problem infeasibility in practice when the scenario tree is not fully branched, but then robust constraint satisfaction is not guaranteed.
The ISpS stability of the fixed horizon multi-stage MPC is retained in this framework by assuming negligible linear control error inside the common terminal region.

Simulation results on three examples demonstrate an effective reduction in the prediction horizon and computational delay, provided the system progressively approaches its setpoint, and that it is not significantly disturbed away from it.
Future avenues for this work include handling economic costs, and the use of parametric sensitivities for both horizon and critical scenario updates to further reduce computation time.
\color{black}
\section*{CRediT Author Statement}
\textbf{Zawadi Mdoe:} Conceptualization, Methodology, Software, Investigation, Writing - original draft, Writing - review \& editing, Visualization.
\textbf{Dinesh Krishnamoorthy:} Conceptualization, Methodology, Writing - review \& editing.
\textbf{Johannes \allowbreak J{\"a}schke:} Conceptualization, Methodology, Writing - original draft, Writing - review \& editing, Supervision.
\section*{Declaration of Competing Interest}
The authors declare that they have no known competing financial interests or personal relationships that could have appeared to influence the work reported in this paper.

\appendix
\bibliographystyle{elsarticle-num-names} 
\bibliography{mybib}





\end{document}